\documentclass[a4paper, 11pt]{article}
\usepackage{mathrsfs, amsmath, amsthm}
\usepackage{subfig, graphicx, color}
\usepackage{array}
\usepackage{cases}
\makeatletter
\def\th@plain{%
  \upshape 
}
\makeatother

\makeatletter
\renewenvironment{proof}[1][\proofname]{\par
  \pushQED{\qed}%
  \normalfont \topsep6\p@\@plus6\p@\relax
  \trivlist
  \item[\hskip\labelsep
        \bfseries
    #1\@addpunct{.}]\ignorespaces
}{%
  \popQED\endtrivlist\@endpefalse
}
\makeatother

\newtheorem{theorem}{Theorem}
\numberwithin{theorem}{section}
\newtheorem{lemma}{Lemma}
\newtheorem{corollary}{Corollary}

\newtheorem{conjecture}{Conjecture}
\newtheorem*{conjecture*}{Conjecture}
\newtheorem{case}{Case}
\newtheorem{subcase}{Subcase}[case]

\newtheorem{claim}{Claim}
\newtheorem{fact}{Fact}

\theoremstyle{definition}

\usepackage[left=25mm,top=25mm,bottom=25mm,right=25mm]{geometry}
\usepackage[T1]{fontenc}
\usepackage[varg]{txfonts}

\usepackage{enumitem}
\usepackage[square, numbers, sort&compress]{natbib}

\ifx\pdfoutput\undefined
 \usepackage[dvipdfm,%
  bookmarks=true,%
  bookmarksnumbered=true, 
  bookmarksopen=true, 
  plainpages=false,%
  pdfpagelabels,%
  colorlinks=true, 
  linkcolor=blue, 
  citecolor=blue,%
  hyperindex=true,
  urlcolor=black,
  pdfborder=001]{hyperref}
\else
 \usepackage[pdftex,%
  bookmarks=true,%
  bookmarksnumbered=true, 
  bookmarksopen=true, 
  plainpages=false,%
  pdfpagelabels,%
  colorlinks=true, 
  linkcolor=black, 
  citecolor=black,%
  anchorcolor=green,
  urlcolor= blue,
  breaklinks=true，
  hyperindex=true,
  pdfborder=001]{hyperref}
\fi

\newcounter{Hcase}
\newcounter{Hclaim}

\newcommand{\resetcounter}{\stepcounter{Hcase}\setcounter{case}{0}\stepcounter{Hclaim}\setcounter{claim}{0}}

\newcommand{\etal}{et~al.\ }

\def\int(#1){\mathrm{int}(#1)}
\def\ext(#1){\mathrm{ext}(#1)}
\def\Int(#1){\mathrm{Int}(#1)}
\def\Ext(#1){\mathrm{Ext}(#1)}
\def\mad(#1){\mathrm{mad}(#1)}
\def\la(#1){\mathrm{la}(#1)}

\begin{document}%
\title{Acyclic edge coloring of graphs}
\author{Tao Wang\,\textsuperscript{a, b, }\footnote{{\tt Corresponding
author: wangtao@henu.edu.cn} }\ \ \ Yaqiong Zhang\,\textsuperscript{b}\\
{\small \textsuperscript{a}Institute of Applied Mathematics}\\
{\small Henan University, Kaifeng, 475004, P. R. China}\\
{\small \textsuperscript{b}College of Mathematics and Information Science}\\
{\small Henan University, Kaifeng, 475004, P. R. China}}
\date{}
\maketitle
\begin{abstract}%
An {\em acyclic edge coloring} of a graph $G$ is a proper edge coloring such that the subgraph induced by any two color classes is a linear forest (an acyclic graph with maximum degree at most two). The {\em acyclic chromatic index} $\chiup_{a}'(G)$ of a graph $G$ is the least number of colors needed in an acyclic edge coloring of $G$. Fiam\v{c}\'{i}k (1978) conjectured that $\chiup_{a}'(G) \leq \Delta(G) + 2$, where $\Delta(G)$ is the maximum degree of $G$. This conjecture is well known as Acyclic Edge Coloring Conjecture (AECC). A graph $G$ with maximum degree at most $\kappa$ is {\em $\kappa$-deletion-minimal} if $\chiup_{a}'(G) > \kappa$ and $\chiup_{a}'(H) \leq \kappa$ for every proper subgraph $H$ of $G$. The purpose of this paper is to provide many structural lemmas on $\kappa$-deletion-minimal graphs. By using the structural lemmas, we firstly prove that AECC is true for the graphs with maximum average degree less than four (\autoref{NMAD4}). We secondly prove that AECC is true for the planar graphs without triangles adjacent to cycles of length at most four, with an additional condition that every $5$-cycle has at most three edges contained in triangles (\autoref{NoAdjacent}), from which we can conclude some known results as corollaries. We thirdly prove that every planar graph $G$ without intersecting triangles satisfies $\chiup_{a}'(G) \leq \Delta(G) + 3$ (\autoref{NoIntersect}). Finally, we consider one extreme case and prove it: if $G$ is a graph with $\Delta(G) \geq 3$ and all the $3^{+}$-vertices are independent, then $\chiup_{a}'(G) = \Delta(G)$. We hope the structural lemmas will shed some light on the acyclic edge coloring problems.
\end{abstract}
\section{Introduction}
All graphs considered are finite, simple and undirected. An {\em acyclic edge coloring} of a graph $G$ is a proper edge coloring such that the subgraph induced by any two color classes is a linear forest (an acyclic graph with maximum degree at most two). The {\em acyclic chromatic index} $\chiup_{a}'(G)$ of a graph $G$ is the least number of colors needed in an acyclic edge coloring of $G$. We denote the minimum and maximum degrees of vertices of $G$ by $\delta(G)$ and $\Delta(G)$, respectively. The {\em degree} of a vertex $v$ in $G$, denoted by $\deg(v)$, is the number of incident edges of $G$. A vertex of degree $k$ is called a $k$-vertex, and a vertex of degree at most or at least $k$ is called a $k^{-}$- or $k^{+}$-vertex, respectively. Let $[\kappa]$ stand for the set $\{1, 2, \dots, \kappa\}$.

Fiam\v{c}\'{i}k \cite{MR526851} stated the following conjecture in 1978, which is well known as Acyclic Edge Coloring Conjecture, and Alon \etal \cite{MR1837021} restated it in 2001.

\begin{conjecture}[AECC]\label{AECC}%
For every graph $G$, we have $\chiup_{a}'(G) \leq \Delta(G) + 2$.
\end{conjecture}

Alon, McDiarmid and Reed \cite{MR1109695} proved that the acyclic chromatic index of a graph $G$ is at most $64\Delta(G)$. Molloy and Reed \cite{MR1715600} improved the upper bound to $16\Delta(G)$. Ndreca \etal \cite{MR2864444} improved the upper bound to $\lceil 9.62(\Delta(G)-1)\rceil$. Recently, Esperet and Parreau \cite{MR3037985} further improved it to $4\Delta(G)$ by using the so-called entropy compression method. Note that $\chiup_{a}'(G) \leq 3$ if $\Delta(G) = 2$. Bur{\v{s}}te{\u\i}n \cite{MR539795} proved that every graph with maximum degree four has an acyclic vertex coloring with five colors. Since an acyclic edge coloring of a graph $G$ is an acyclic vertex coloring of its line graph $L(G)$, and the maximum degree of a line graph $L(G)$ of a subcubic graph $G$ is at most four, it follows that $\chiup_{a}'(G) \leq 5$ if $\Delta(G) = 3$. Hence, \autoref{AECC} holds for $\Delta(G) \leq 3$. Furthermore, Andersen \etal \cite{MR2988879} proved that the acyclic chromatic index of a connected subcubic graph $G$ is at most four unless $G$ is $K_{4}$ or $K_{3, 3}$, the acyclic chromatic index of $K_{4}$ and $K_{3, 3}$ is five. \autoref{AECC} has also been verified for some special classes of graphs. Muthu \etal \cite{Muthu2007} proved that $\chiup_{a}'(G) \leq \Delta(G) + 1$ for every outerplanar graph $G$. Hou \etal \cite{MR2656747} proved that $\chiup_{a}'(G) = \Delta(G)$ for every outerplanar graph $G$ with $\Delta(G) \geq 5$.
The conjecture is also true for planar graphs with girth at least five \cite{MR2491777, MR2601248} and planar graphs with girth at least four \cite{MR3065111}.

Fiedorowicz \etal \cite{MR2458434} proved that $\chiup_{a}'(G) \leq 2\Delta(G) + 29$ for every planar graph $G$. Basavaraju \etal \cite{MR2817509} showed that the acyclic chromatic index of a planar graph $G$ is at most $\Delta(G) + 12$. Furthermore, Guan \etal \cite{MR3031510} improved the upper bound to $\Delta(G) + 10$ and Wang \etal \cite{MR2994403} further improved it to $\Delta(G) + 7$.

The {\em maximum average degree} $\mad(G)$ of a graph $G$ is the largest average degree of its subgraphs, that is,
\[\mad(G) = \max_{H \subseteq G} \left\{\frac{2|E(H)|}{|V(H)|}\right\}.\]

For the graph with small maximum average degree, we have known the following result.

\begin{theorem}[Basavaraju and Chandran \cite{MR2988880}]%
If $G$ is a graph with $\mad(G) < 4$, then $\chiup_{a}'(G) \leq \Delta(G) + 3$.
\end{theorem}

Recently, Wang \etal \cite{MR2979486} and Hou \cite{2012arXiv1202.6129H} independently proved the following result.

\begin{theorem}\label{MAD4}%
If $G$ is a graph with $\mad(G) < 4$, then $\chiup_{a}'(G) \leq \Delta(G) + 2$.
\end{theorem}

A graph $G$ with maximum degree at most $\kappa$ is {\em $\kappa$-deletion-minimal} if $\chiup_{a}'(G) > \kappa$ and $\chiup_{a}'(H) \leq \kappa$ for every proper subgraph $H$ of $G$. A graph property $\mathcal{P}$ is {\em deletion-closed} if $\mathcal{P}$ is closed under taking subgraphs.

In section \ref{SL}, we provide many structural lemmas on $\kappa$-deletion-minimal graphs. In section \ref{MResults}, we firstly prove that AECC is true for the graphs with maximum average degree less than four (\autoref{NMAD4}). We secondly prove that AECC is true for the planar graphs without triangles adjacent to cycles of length at most four, with an additional condition that every $5$-cycle has at most three edges contained in triangles (\autoref{NoAdjacent}), from which we can conclude some known results as corollaries. We thirdly prove that every planar graph $G$ without intersecting triangles satisfies $\chiup_{a}'(G) \leq \Delta(G) + 3$ (\autoref{NoIntersect}). In section \ref{ConcludingR}, we consider one extreme case and prove it: if $G$ is a graph with $\Delta(G) \geq 3$ and all the $3^{+}$-vertices are independent, then $\chiup_{a}'(G) = \Delta(G)$.

\section{Preliminary}
Let $G$ be a graph and $H$ be a subgraph of $G$. An acyclic edge coloring of $H$ is a {\em partial acyclic edge coloring} of $G$. Let $\mathcal{U}_{\phi}(v)$ denote the set of colors which are assigned to the edges incident with $v$ with respect to $\phi$. Let $C_{\phi}(v) = [\kappa] \setminus \mathcal{U}_{\phi}(v)$ and $\Upsilon_{\phi}(uv) = \mathcal{U}_{\phi}(v) \setminus \{\phi(uv)\}$. Let $W_{\phi}(uv) = \{u_{i} \mid uu_{i} \in E(G) \mbox{ and } \phi(uu_{i}) \in \Upsilon_{\phi}(uv)\}$. Notice that $W_{\phi}(uv)$ may be not same with $W_{\phi}(vu)$. An $(\alpha, \beta)$-maximal dichromatic path with respect to $\phi$ is a maximal path whose edges are colored by $\alpha$ and $\beta$ alternately. An $(\alpha, \beta, u, v)$-critical path with respect to $\phi$ is an $(\alpha, \beta)$-maximal dichromatic path which starts at $u$ with color $\alpha$ and ends at $v$ with color $\alpha$. An $(\alpha, \beta, u, v)$-alternating path with respect to $\phi$ is an $(\alpha, \beta)$-dichromatic path starting at $u$ with color $\alpha$ and ending at $v$ with color $\beta$.

Let $\phi$ be a partial acyclic edge coloring of $G$. A color $\alpha$ is {\em candidate} for an edge $e$ in $G$ with respect to a partial edge coloring of
$G$ if none of the adjacent edges of $e$ is colored with $\alpha$.  A candidate color $\alpha$ is {\em valid} for an edge $e$ if assigning the color $\alpha$ to $e$ does not result in any dichromatic cycle in $G$.

\begin{fact}[\cite{MR2817509}]%
Given partial acyclic edge coloring of $G$ and two colors $\alpha, \beta$, there exists at most one $(\alpha, \beta)$-maximal path containing a particular vertex $v$. \qed
\end{fact}

\begin{fact}[\cite{MR2817509}]%
Let $G$ be a $\kappa$-deletion-minimal graph and $uv$ be an edge of $G$. If $\phi$ is an acyclic edge coloring of $G - uv$, then no candidate color for $uv$ is valid. Furthermore, if $\mathcal{U}(u) \cap \mathcal{U}(v) = \emptyset$, then $\deg(u) + \deg(v) = \kappa + 2$; if $|\mathcal{U}(u) \cap \mathcal{U}(v)| = s$, then $\deg(u) + \deg(v) + \sum\limits_{w \in W(uv)} \deg(w) \geq \kappa+2s+2$. \qed
\end{fact}

We remind the readers that we will use these two facts frequently, so please keep these in mind and we will not refer it at every time. In the following sections, if there is no confusion, we omit the subscript $\phi$. When we say ``acyclic edge coloring'' it means acyclic edge coloring with at most $\kappa$ colors.
\section{Structural lemmas}\label{SL}
In this section, we provide many structural lemmas on $\kappa$-deletion-minimal graphs. Throughout this section, we assume that $G$ is a $\kappa$-deletion-minimal graph and $\kappa$ is an integer.

\begin{lemma}\label{kappa=2}%
If $G$ is a $\kappa$-deletion-minimal graph, then $G$ is $2$-connected.
\end{lemma}

\begin{lemma}[Hou \etal \cite{MR2849391}]\label{DegreeSum}
If $G$ is a $\kappa$-deletion-minimal graph and $w_{0}$ is a vertex in $G$, then
\[
\sum_{w \in N_{G}(w_{0})} \deg(w) \geq \kappa + \deg(w_{0}).
\]
\end{lemma}

\begin{lemma}\label{2+edge}%
Let $G$ be a $\kappa$-deletion-minimal graph. If $v$ is adjacent to a $2$-vertex $v_{0}$ and $N_{G}(v_{0}) = \{w, v\}$, then $v$ is adjacent to at least $\kappa - \deg(w) + 1$ vertices of degree at least $\kappa - \deg(v) + 2$. Moreover,
\begin{enumerate}[label=(\Alph*)]%
\item\label{A} if $\kappa \geq \deg(v) + 1$ and $wv \in E(G)$, then $v$ is adjacent to at least $\kappa - \deg(w) + 2$ vertices of degree at least $\kappa - \deg(v) + 2$, and $\deg(v) \geq \kappa - \deg(w) + 3$;
\item\label{B} if $\kappa \geq \Delta(G) + 2$ and $v$ is adjacent to precisely $\kappa - \Delta(G) + 1$ vertices of degree at least $\kappa - \Delta(G) + 2$, then $v$ is adjacent to at most $\deg(v) + \Delta(G) - \kappa - 3$ vertices of degree two and $\deg(v) \geq \kappa - \Delta(G) + 4$.
\end{enumerate}
\end{lemma}
\begin{proof}%
Since $G$ is $\kappa$-deletion-minimal, it follows that $G - v_{0}$ admits an acyclic edge coloring $\phi$. Without loss of generality, assume that $N_{G}(v) = \{v_{0}, v_{1}, \dots, v_{n}\}$ and $\phi(vv_{i}) = i$ for $1 \leq i \leq n$. If $\kappa = \deg(v)$, then \autoref{DegreeSum} applies. So we may assume that $\kappa > \deg(v)$ and $|C(v)| \geq 2$.

Suppose that $C(w) \cap C(v) \neq \emptyset$. Choose colors $\lambda_{1} \in C(w) \cap C(v)$ and $\lambda_{2} \in C(v) \setminus \{\lambda_{1}\}$. If assigning $\lambda_{1}$ to $wv_{0}$ and $\lambda_{2}$ to $vv_{0}$, we obtain an acyclic edge coloring of $G$, a contradiction. So we may assume that $C(w) \cap C(v) = \emptyset$. Thus $C(w) \subseteq \mathcal{U}(v)$ and $C(v) \subseteq \mathcal{U}(w)$. So we may assume that $C(w) = \{1, \dots, m\}$, where $m = \kappa -\deg(w) + 1$.

If there exists no $(\lambda, \lambda^{*}, v, w)$-alternating path with $\lambda \in C(w)$ and $\lambda^{*} \in C(v)$, then assigning $\lambda$ to $wv_{0}$ and $\lambda^{*}$ to $vv_{0}$ results in an acyclic edge coloring of $G$, a contradiction.

Hence, there exists an $(i, j, v, w)$-alternating path, where $i$ and $j$ are arbitrarily chosen from $C(w)$ and $C(v)$ respectively. Thus, we have $\{n+1, \dots, \kappa\} = C(v) \subseteq \Upsilon(vv_{i})$ for $1 \leq i \leq m$. Recall that $|C(w)| = m = \kappa -\deg(w) + 1$ and $|C(v)| = \kappa -\deg(v) + 1$, it follows that the vertex $v$ is adjacent to at least $\kappa - \deg(w) + 1$ vertices of degree at least $\kappa - \deg(v) + 2$.

(A) By \autoref{DegreeSum}, we have $\deg(w) \geq \kappa + 2 - \deg(v)$ and $w$ is a $(\kappa - \deg(v) + 2)^{+}$-vertex. From the above arguments, if $wv \in E(G)$, then $w \notin \{v_{1}, v_{2}, \dots, v_{m}\}$ and \ref{A} hold.

(B) Suppose that $\kappa \geq \Delta(G) + 2$ and $v$ is adjacent to precisely $\kappa - \Delta(G) + 1$ vertices of degree at least $\kappa - \Delta(G) + 2$. It follows that $\deg(w) = \Delta(G)$ and the precisely $\kappa - \Delta(G) + 1$ vertices of degree at least $\kappa - \Delta(G) + 2$ are $v_{1}, v_{2}, \dots, v_{m}$. By contradiction, we may assume that $v_{i}$ is a $2$-vertex and $N_{G}(v_{i}) = \{v, w_{i}\}$ for $m + 2 \leq i \leq n$. Note that $\{v_{0}, v_{m+2}, \dots, v_{n}\}$ is an independent set in $G$ by the $2$-connectivity, thus $w \notin \{v_{m+2}, \dots, v_{n}\}$.

\begin{claim}\label{VWPath}%
There exists an $(i, j, v, w)$-alternating path, where $i$ and $j$ are arbitrarily chosen from $C(w)$ and $\{m+1, \dots, n\}$.
\end{claim}
\begin{proof}
By symmetry, suppose that there exists no $(1, \alpha, v, w)$-alternating path. Removing the color $\alpha$ from $vv_{\alpha}$, assigning $1$ to $wv_{0}$ and $\alpha$ to $vv_{0}$, we obtain an acyclic edge coloring $\psi$ of $G - vv_{\alpha}$. By Fact 2 and $\deg(v) + \deg(v_{\alpha}) \leq \deg(v) + (\kappa - \deg(v) + 1) < \kappa + 2$, we have $\mathcal{U}_{\psi}(v) \cap \mathcal{U}_{\psi}(v_{\alpha}) \neq \emptyset$.

Suppose that $\Upsilon(vv_{\alpha}) \cap \{m+1, \dots, n\} = \emptyset$. We can extend $\psi$ by assigning a color in $C(v) \setminus \Upsilon(vv_{\alpha})$ to $vv_{\alpha}$, and thus obtain an acyclic edge coloring of $G$, a contradiction. Hence, we may assume that $\Upsilon(vv_{\alpha}) \cap \{m+1, \dots, n\} \neq \emptyset$. If $\deg(v_{\alpha}) = 2$, then $\Upsilon(vv_{\alpha}) = \{\alpha^{*}\} \subseteq \{m+1, \dots, n\}$, but every color in $\{n+1, \dots, \kappa\} \setminus \Upsilon(vv_{\alpha^{*}})$ is valid for $vv_{\alpha}$ with respect to $\psi$,  a contradiction. Hence, $v_{\alpha}$ must be a $3^{+}$-vertex and $\alpha = m + 1$. By symmetry, we may assume that $n \in \Upsilon(vv_{m+1})$. By Fact 2, we have $\Upsilon(vv_{m+1}) \cap \{1, \dots, m\} \neq \emptyset$. Hence, $\{n + 1, \dots, \kappa\} \setminus (\Upsilon(vv_{m+1}) \cup \{\phi(v_{i}w_{i}) \mid v_{i} \in W(vv_{\alpha}) \mbox{ and } i \geq m+2\}) \neq \emptyset$, but every color in this set is valid for $vv_{m + 1}$ with respect to $\psi$, a contradiction.
\end{proof}
Hence, $\deg(v_{i}) = \Delta(G)$ and $\Upsilon(vv_{i}) = \{m+1, \dots, \kappa\}$ for $1 \leq i \leq m$. Let $\phi_{i,j}$ be the proper edge coloring obtained from $\phi$ by exchanging the colors on $vv_{i}$ and $vv_{j}$ for $1 \leq i < j \leq m$.

\begin{claim}\label{Phi12}%
There exist $i_{0}$ and $j_{0}$ in $\{1, \dots, m\}$ such that the edge coloring $\phi_{i_{0}, j_{0}}$ has no dichromatic cycle containing the edge $vv_{m+1}$.
\end{claim}
\begin{proof}
The vertex $v_{m+1}$ is a vertex of degree at most $m$, thus there are at most $m-1$ critical paths containing $vv_{m+1}$ and passing through $v_{m+1}$. Hence, there exists a vertex in $\{v_{1}, v_{2}, \dots, v_{m}\}$, say $v_{1}$, such that there exists no critical path with respect to $\phi$ starting at $vv_{m+1}$ and ending at $v_{1}$. Therefore, $\phi_{1,2}$ is the desired edge coloring if there is no $(m+1, 1, v, v_{2})$-critical path with respect to $\phi$. Otherwise, $\phi_{1,3}$ is the desired edge coloring (note that $m \geq 3$).
\end{proof}

Without loss of generality, we may assume that $\phi_{1, 2}$ is the edge coloring obtained in \autoref{Phi12}. If $\phi_{1, 2}$ is an acyclic edge coloring of $G - v_{0}$, then we extend $\phi_{1, 2}$ by assigning $\kappa$ to $vv_{0}$ and $1$ to $wv_{0}$, and thus obtain an acyclic edge coloring of $G$, which is a contradiction. Hence, the proper edge coloring $\phi_{1, 2}$ is not an acyclic edge coloring of $G - v_{0}$, and then it admits dichromatic cycles containing $vv_{1}$ or $vv_{2}$ but not containing $vv_{m+1}$. Let $T_{1} = \{\theta \mid \mbox{ there exists a $(1, \theta)$-dichromatic cycle containing $vv_{2}$ with respect to $\phi_{1, 2}$}\}$. Let $T_{2} = \{\theta \mid \mbox{ there exists a $(2, \theta)$-dichromatic cycle containing $vv_{1}$ with respect to $\phi_{1, 2}$}\}$. Consequently, we have $T_{1} \cup T_{2} \subseteq \{m+2, \dots, \kappa\}$ and $T_{1} \cup T_{2} \neq \emptyset$. Note that $T_{1} \cap T_{2} = \emptyset$, since all the vertices $v_{i}$ with $i \geq m+2$ are $2$-vertices.

\begin{claim}%
We may assume that $|T_{1}| \leq 1$ and $|T_{2}| \leq 1$.
\end{claim}
\begin{proof}
Without loss of generality, we may assume that $T_{1} = \{m+2, \dots, r\}$ and $T_{2} = \{s, \dots, n\}$, where $r < s$ (note that $T_{1} \cap T_{2} = \emptyset$). If $|T_{1}| > 1$, then reassign $i+1$ to $vv_{i}$ for $m+2 \leq i \leq r-1$ and reassign $m+2$ to $vv_{r}$. Similarly, if $|T_{2}| > 1$, then reassign $j+1$ to $vv_{j}$ for $s \leq j \leq n-1$ and reassign $s$ to $vv_{n}$. Finally, we obtain a proper edge coloring $\phi^{*}$ of $G - v_{0}$ satisfying $|T_{1}| \leq 1$ and $|T_{2}| \leq 1$ with respect to $\phi^{*}$, but it has no dichromatic cycle containing $vv_{m+1}$.
\end{proof}

By symmetry, we may assume that $|T_{1}| \geq |T_{2}|$ and $T_{1} = \{m+2\}$. The following proof is divided into two cases.
\begin{enumerate}[leftmargin=*, label = {\bf Case B\arabic*}]%
\item\label{C1} $|T_{2}| = 0$.

We modify $\phi_{1, 2}$ by assigning $1$ to $wv_{0}$ and $m+2$ to $vv_{0}$, and removing the colors on $vv_{m+2}$ and $v_{m+2}w_{m+2}$, and thus obtain an acyclic edge coloring $\psi^{*}$ of $G - v_{m+2}$. By similar arguments as above, we have that $\deg(w_{m+2}) = \Delta(G)$ and $C_{\psi^{*}}(w_{m+2}) = \{1, \dots, m\}$. We extend $\psi^{*}$ by reassigning $\kappa$ to $vv_{m+2}$ and $3$ to $v_{m+2}w_{m+2}$ (note that $w_{m+2} \neq w$ since $1 \notin \mathcal{U}(w)$ and $1 \in \mathcal{U}(w_{m+2})$), and thus obtain an acyclic edge coloring of $G$, a contradiction.

\item $|T_{2}| = 1$.

Without loss of generality, we may assume that $T_{2} = \{m+3\}$. Obviously, $\phi(v_{m+2}w_{m+2}) = 1$ and $\phi(v_{m+3}w_{m+3}) = 2$. Note that $w \neq w_{m+2}$ and $w \neq w_{m+3}$ since $C(w) = \{1, \dots, m\}$ but $1 \in \mathcal{U}(w_{m+2})$ and $2 \in \mathcal{U}(w_{m+3})$. Suppose that there is a color $\theta$ in $C(w_{m+3})$ with $\theta \in \{3, \dots, m\} \cup \{n+1, \dots, \kappa\}$. We can modify $\phi_{1, 2}$ by reassigning $\theta$ to $v_{m+3}w_{m+3}$, the resulting edge coloring has similar properties as $\phi_{1, 2}$ and then we go back to \ref{C1}. Hence, $C(w_{m+3}) \subseteq \{1\} \cup \{m+1, \dots, n\}$, so we may assume that there exists a color $\theta$ in $C(w_{m+3})$ with $m+1 \leq \theta \leq n$. Since $\deg(v) + \deg(v_{m+3}) + \deg(v_{\theta}) \leq \deg(v) + 2 + (\kappa - \deg(v) + 1) < \kappa + 4$, we can modify $\phi_{1, 2}$ by reassigning $\theta$ to $v_{m+3}w_{m+3}$ and reassigning a suitable color to $vv_{m+3}$, such that the resulting proper edge coloring has no dichromatic cycle containing $vv_{m+3}$, then we go back to \ref{C1}.
\resetcounter\qedhere
\end{enumerate}
\end{proof}

\begin{lemma}\label{2++edge}%
Let $G$ be a $\kappa$-deletion-minimal graph with $\kappa \geq \Delta(G) + 2$. If $v_{0}$ is a $2$-vertex, then every neighbor of $v_{0}$ has degree at least $\kappa - \Delta(G) + 4$.
\end{lemma}
\begin{proof}%
Let $N_{G}(v_{0}) = \{w, v\}$. By contradiction and \autoref{2+edge}, we can suppose that $v$ is a $(\kappa - \Delta(G) + 3)$-vertex and $N_{G}(v) = \{v_{0}, v_{1},\dots, v_{n}\}$, where $n = \kappa - \Delta(G) + 2 \geq 4$. Since $G$ is $\kappa$-deletion-minimal, it follows that $G - v_{0}$ admits an acyclic edge coloring $\phi$ with $\phi(vv_{i}) = i$ for $1 \leq i \leq n$, and $C(v) = \{n+1, \dots, \kappa\}$.

Suppose that $C(w) \cap C(v) \neq \emptyset$. Choose colors $\lambda_{1} \in C(w) \cap C(v)$ and $\lambda_{2} \in C(w) \setminus \{\lambda_{1}\}$, and extend $\phi$ by assigning $\lambda_{1}$ to $vv_{0}$ and $\lambda_{2}$ to $wv_{0}$, and thus we obtain an acyclic edge coloring of $G$, a contradiction. So we may assume that $C(w) \cap C(v) = \emptyset$. Consequently, $C(w) \subseteq \{1, \dots, n\}$ and $\{n+1, \dots, \kappa\} \subseteq \mathcal{U}(w)$. So we may assume that $\{1, \dots, n-1\} \subseteq C(w) \subseteq \{1, \dots, n\}$.

If there exists no $(\lambda, \lambda^{*}, v, w)$-alternating path with $\lambda \in C(w)$ and $\lambda^{*} \in C(v)$, then we extend $\phi$ by assigning $\lambda$ to $wv_{0}$ and $\lambda^{*}$ to $vv_{0}$, thus obtain an acyclic edge coloring of $G$, a contradiction.

Hence, there exists an $(i, j, v, w)$-alternating path, where $i$ and $j$ are arbitrarily chosen from $C(w)$ and $C(v)$ respectively. Thus, we have $\{n+1, \dots, \kappa\} = C(v) \subseteq \Upsilon(vv_{i})$ for $i \in C(w)$. 

Let $\phi_{i, j}$ denote the edge coloring obtained from $\phi$ by exchanging the colors on $vv_{i}$ and $vv_{j}$ for $1 \leq i < j \leq n$. Note that $1 \leq i \leq n-1$ and $i \in C(w)$.

\begin{claim}\label{PAEC}%
If the edge coloring $\phi_{i, j}$ is a proper edge coloring of $G - v_{0}$, then it must contain dichromatic cycles.
\end{claim}
\begin{proof}%
If $\phi_{i, j}$ is an acyclic edge coloring of $G - v_{0}$, then we extend $\phi_{i, j}$ by assigning $i$ to $wv_{0}$ and $\kappa$ to $vv_{0}$, and obtain an acyclic edge coloring of $G$, a contradiction.
\end{proof}

\begin{claim}%
$\{n+1, \dots, \kappa\} \subseteq \mathcal{U}(v_{n})$.
\end{claim}
\begin{proof}
If $C(w) = \{1, 2, \dots, n\}$, then we have finished it. By contradiction and symmetry, we may assume that $C(w) = \{1, 2, \dots, n-1\}$ and $\kappa \notin \mathcal{U}(v_{n})$. The fact that $\kappa \in \mathcal{U}(w)$ and $\kappa \notin \mathcal{U}(v_{n})$ imply $w \neq v_{n}$. Reassigning $n, \kappa$ and an arbitrary color $t$ in $C(w)$ to $vv_{0}, vv_{n}$ and $wv_{0}$ respectively, yields a proper edge coloring of $G$, which then must contain a $(n, t)$-dichromatic cycle containing $v_{0}$, otherwise it is an acyclic edge coloring of $G$, a contradiction. So we have that $n \in \mathcal{U}(v_{t})$ and $\Upsilon(vv_{t}) = \{n, \dots, \kappa\}$ for $1 \leq t \leq n - 1$. Hence, the edge colorings $\phi_{1, 2}, \phi_{2, 3}$ and $\phi_{1, 3}$ are proper edge colorings of $G - v_{0}$, and then there exist dichromatic cycles containing $vv_{n}$ with respect to each of these edge colorings by \autoref{PAEC}. Without loss of generality, we may assume that there exists a $(n, 2)$-dichromatic cycle with respect to $\phi_{1, 2}$; in other words, there exists a $(n, 2, v, v_{1})$-critical path with respect to $\phi$. Hence, there is no $(n, 2, v, v_{t})$-critical path with respect to $\phi$ for $2 \leq t \leq n-1$. There exists a $(n, 3)$-dichromatic cycle containing $vv_{n}$ with respect to $\phi_{2, 3}$, and then there exists a $(n, 3, v, v_{2})$-critical path with respect to $\phi$. Hence, there exists no $(n, 3)$-dichromatic cycle containing $vv_{n}$ with respect to $\phi_{1, 3}$, and thus there is a $(n, 1, v, v_{3})$-critical path with respect to $\phi$. Reassigning $3, 1, 2$ to $vv_{1}, vv_{2}, vv_{3}$ respectively, and assigning $1$ to $wv_{0}$ and $\kappa$ to $vv_{0}$, we obtain an acyclic edge coloring of $G$, a contradiction.
\end{proof}

Now, we have $\{n+1, \dots, \kappa\} \subseteq \Upsilon(vv_{i})$ for $1 \leq i \leq n$, and it follows that $|\Upsilon(vv_{i}) \cap \{1, \dots, n\}| \leq 1$ for $1 \leq i \leq n$.

\begin{claim}\label{ClaimE}%
If $\Upsilon(vv_{n}) \cap \{1, \dots, n\} = \emptyset$, then $\Upsilon(vv_{i}) = \{n, \dots, \kappa\}$ for $1 \leq i \leq n-1$.
\end{claim}
\begin{proof}
By contradiction and symmetry, assume that $n \notin \Upsilon(vv_{n-1})$. By \autoref{PAEC}, the proper edge coloring $\phi_{n-1, n}$ must contain dichromatic cycles. Note that there is no dichromatic cycle containing $vv_{n}$ with respect to $\phi_{n-1, n}$. It follows that there exists a $(n, i)$-dichromatic cycle containing $vv_{n-1}$ with respect to $\phi_{n-1, n}$, where $i \leq n - 2$. By symmetry, assume that there exists a $(n, 1)$-dichromatic cycle containing $vv_{n-1}$ with respect to $\phi_{n-1, n}$. Hence, $\Upsilon(vv_{n-1}) \cap \{1, \dots n\} = \{1\}$ and $\Upsilon(vv_{1}) \cap \{1, \dots, n\} = \{n\}$. If there exists a vertex $v_{j}$ with $2 \leq j \leq n-2$ such that $n \in \Upsilon(vv_{j})$, then $\phi_{1, j}$ is an acyclic edge coloring of $G - v_{0}$, which contradicts \autoref{PAEC}. Hence, we have $n \notin \Upsilon(vv_{j})$ for $2 \leq j \leq n-2$. Since there exists a $(1, n, v, v_{n-1})$-critical path with respect to $\phi$, thus $\phi_{2, n}$ is an acyclic edge coloring of $G - v_{0}$, which contradicts \autoref{PAEC} again.
\end{proof}

By \autoref{ClaimE}, if $\Upsilon(vv_{n}) \cap \{1, \dots, n\} =\emptyset$, then $\Upsilon(vv_{i}) = \{n, \dots, \kappa\}$ for $1 \leq i \leq n-1$, and then $\phi_{1, 2}$ is an acyclic edge coloring of $G - v_{0}$, which contradicts \autoref{PAEC}. So we may assume that $\Upsilon(vv_{n}) \cap \{1, \dots, n\} \neq \emptyset$. By symmetry, assume that $n-1 \in \Upsilon(vv_{n})$.

\begin{case}%
There exist two vertices $v_{i}$ and $v_{j}$ with $1 \leq i < j \leq n-2$ such that $n \in \Upsilon(vv_{i}) \cap \Upsilon(vv_{j})$. 
\end{case}
Without loss of generality, we may assume that $n \in \mathcal{U}(v_{1}) \cap \mathcal{U}(v_{2})$. The edge coloring $\phi_{1, 2}$ is an acyclic edge coloring of $G - v_{0}$, which contradicts \autoref{PAEC}.
\begin{case}%
For every vertex $v_{i}$ with $1 \leq i \leq n-2$, we have $n \notin \Upsilon(vv_{i})$.
\end{case}

If $1 \in \mathcal{U}(v_{n-1})$, then the proper edge coloring $\phi_{2, n}$ is an acyclic edge coloring of $G - v_{0}$, which contradicts \autoref{PAEC}. So we may assume that $1 \notin \mathcal{U}(v_{n-1})$. By \autoref{PAEC}, the proper edge coloring $\phi_{1, n}$ must contain $(n, n-1)$-dichromatic cycle containing $vv_{1}$ and $n \in \Upsilon(vv_{n-1})$, but $\phi_{2, n}$ is an acyclic edge coloring of $G - v_{0}$, which contradicts \autoref{PAEC}.

\begin{case}%
There is only one vertex $v_{i}$ with $1 \leq i \leq n-2$ such that $n \in \Upsilon(vv_{i})$. By symmetry, we may assume that $n \in \mathcal{U}(v_{2})$.
\end{case}

\begin{subcase}%
$\{1, n\} \subseteq C(v_{n-1})$.
\end{subcase}

By \autoref{PAEC}, the proper edge coloring $\phi_{1, n}$ must contain $(n, 2)$-dichromatic cycle containing $vv_{1}$. Hence, $2 \in \Upsilon(vv_{1})$. If $n = 4$, then $\phi_{1, 3}$ is an acyclic edge coloring of $G - v_{0}$, which contradicts \autoref{PAEC}. So we may assume that $n \geq 5$. The proper edge coloring $\phi_{1, n-1}$ must contain a dichromatic cycle containing $vv_{n-1}$, say $(1, i)$-dichromatic cycle, where $i \in \{3, \dots, n-2\}$. By symmetry, we may assume that there is a $(1, 3)$-dichromatic cycle with respect to $\phi_{1, n-1}$. Hence, $1 \in \Upsilon(vv_{3})$ and $3 \in \Upsilon(vv_{n-1})$. Reassigning $3, 1, 2$ to $vv_{1}, vv_{2}, vv_{3}$ respectively, and assigning $2$ to $wv_{0}$ and $\kappa$ to $vv_{0}$, we obtain an acyclic edge coloring of $G$, a contradiction.

\begin{subcase}%
$\{1, n\} \cap \Upsilon(vv_{n-1}) \neq \emptyset$.
\end{subcase}
Suppose that $2 \in \mathcal{U}(v_{1})$. Reassigning $n-1, 1$ and $2$ to $vv_{1}, vv_{2}$ and $vv_{n-1}$ respectively results in an acyclic edge coloring of $G - v_{0}$, and extending it by assigning $2$ to $wv_{0}$ and $\kappa$ to $vv_{0}$, we obtain an acyclic edge coloring of $G$, a contradiction. So we may assume that $2 \notin \mathcal{U}(v_{1})$. By \autoref{PAEC}, the proper edge coloring $\phi_{1, 2}$ admits a dichromatic cycle containing $vv_{1}$. Thus, $n \geq 5$. So we may assume that there is a $(2, 3)$-dichromatic cycle with respect to $\phi_{1, 2}$, which implies that $3 \in \Upsilon(vv_{1})$ and $2 \in \Upsilon(vv_{3})$. Reassigning $2, 3, 1$ to $vv_{1}, vv_{2}, vv_{3}$ respectively, and assigning $2$ to $wv_{0}$ and $\kappa$ to $vv_{0}$, we obtain an acyclic edge coloring of $G$, a contradiction.
\resetcounter
\end{proof}

\begin{lemma}\label{24edge}%
Let $G$ be a $\kappa$-deletion-minimal graph with $\kappa \geq \Delta(G) + 1$. If $v_{0}$ is a $2$-vertex, then every neighbor of $v_{0}$ has degree at least four.
\end{lemma}
\begin{proof}%
Let $N_{G}(v_{0}) = \{w, v\}$. By contradiction and \autoref{2+edge}, suppose that $v$ is a $3$-vertex and $N_{G}(v) = \{v_{0}, v_{1}, v_{2}\}$. Since $G$ is $\kappa$-deletion-minimal, it follows that $G - v_{0}$ admits an acyclic edge coloring $\phi$. Without loss of generality, assume that $\phi(vv_{1}) = 1$ and $\phi(vv_{2}) = 2$, and then $C(v) = \{3, 4, \dots, \kappa\}$.

Suppose that $C(w) \cap \{3, \dots, \kappa\} \neq \emptyset$. Choose colors $\lambda_{1} \in C(w) \cap \{3, \dots, \kappa\}$ and $\lambda_{2} \in C(w) \setminus \{\lambda_{1}\}$, and extend $\phi$ by assigning $\lambda_{1}$ to $vv_{0}$ and $\lambda_{2}$ to $wv_{0}$, we obtain an acyclic edge coloring of $G$, a contradiction. So we may assume that $C(w) \cap \{3, \dots, \kappa\} = \emptyset$. Consequently, $C(w) = \{1, 2\}$ and $\deg(w) = \Delta(G)$. Hence, there exists an $(i, j, v, w)$-critical path, where $i \in \{1, 2\}$ and $j \in \{3, \dots, \kappa\}$, otherwise, assigning $i$ to $wv_{0}$ and $j$ to $vv_{0}$ results in an acyclic edge coloring of $G$, a contradiction. Thus, we have $\deg(v_{i}) = \Delta(G)$ and $\Upsilon(vv_{i}) = \{3, \dots, \kappa\}$ for $i \in \{1, 2\}$. Exchanging the colors on $vv_{1}$ and $vv_{2}$ results in a new acyclic edge coloring of $G - v_{0}$. Extending this edge coloring by assigning $1$ to $wv_{0}$ and $\kappa$ to $vv_{0}$, we obtain an acyclic edge coloring of $G$, a contradiction.
\end{proof}
\begin{lemma}\label{Good-3-vertex}%
Let $G$ be a $\kappa$-deletion-minimal graph with $\kappa \geq \Delta(G) + 2$ and $v$ be a $3$-vertex with $N_{G}(v) = \{w, v_{1}, v_{2}\}$. If $\deg(w) = \kappa - \Delta(G) + 2$, then $G$ has the following properties:

\begin{enumerate}[label= (\alph*)]%
\item\label{3a} there is exactly one common color at $w$ and $v$ for any acyclic edge coloring of $G - wv$. By symmetry, we may assume that the color on $vv_{1}$ is the common color;
\item\label{3b} $\deg(v_{1}) = \Delta(G) \geq \deg(v_{2}) \geq \kappa -\Delta(G) + 3$;
\item\label{3c} the edge $wv$ is not contained in any triangle in $G$ and $w$ is adjacent to exactly one $3^{-}$-vertex, say $v$;
\item\label{3d} the vertex $v_{1}$ is adjacent to at least $\kappa - \deg(v_{2}) + 1$ vertices of degree at least $\kappa - \Delta(G) + 2$;
\item\label{3e} the vertex $v_{2}$ is adjacent to at least $\kappa - \Delta(G)$ vertices of degree at least $\kappa - \deg(v_{2}) + 2$;
\item\label{3f} the vertex $v_{2}$ is adjacent to at least $\kappa - \Delta(G) + 1$ vertices of degree at least four.
\end{enumerate}
\end{lemma}
\begin{proof}%
Let $N_{G}(w) = \{v, w_{1}, \dots, w_{n}\}$, where $n = \kappa - \Delta(G) + 1 \geq 3$. Since $G$ is $\kappa$-deletion-minimal, it follows that $G - wv$ admits an acyclic edge coloring $\phi$ with $\phi(ww_{i}) = i$ for $1 \leq i \leq n$. Since $\deg(w)+ \deg(v) = \deg(w) + 3 \neq \kappa + 2$, Fact 2 guarantees $|\mathcal{U}(w) \cap \mathcal{U}(v)| \geq 1$. Without loss of generality, assume that $\phi(vv_{1}) = 1$.

\begin{claim}%
$|\mathcal{U}(w) \cap \mathcal{U}(v)| = 1$.
\end{claim}
\begin{proof}%
By contradiction and symmetry, we may assume that $\phi(vv_{2}) = 2$. For any $\alpha$ with $n+1 \leq \alpha \leq \kappa$, there exists a $(1, \alpha, v, w)$-critical path or there exists a $(2, \alpha, v, w)$-critical path. Let
\begin{align*}%
T_{1} = \{\alpha_{1} \mid \alpha_{1} \in \{n+1, \dots, \kappa\} \mbox{ and there exists a $(1, \alpha_{1}, v, w)$-critical path}\}, \\
T_{2} = \{\alpha_{2} \mid \alpha_{2} \in \{n+1, \dots, \kappa\} \mbox{ and there exists a $(2, \alpha_{2}, v, w)$-critical path}\}.
\end{align*}
Hence, $T_{1} \cup T_{2} = \{n+1, \dots, \kappa\}$.

\begin{case}%
Either $\Upsilon(vv_{1}) \nsupseteq \{n+1, \dots, \kappa\}$ or $\Upsilon(vv_{2}) \nsupseteq \{n+1, \dots, \kappa\}$.
\end{case}
By symmetry, we may assume that $n+1 \notin \Upsilon(vv_{2})$. It follows that there exists a $(1, n+1, v, w)$-critical path. Reassigning $n+1$ to $vv_{2}$ results in a new acyclic edge coloring $\psi$ of $G - vw$ with $|\mathcal{U}_{\psi}(w) \cap \mathcal{U}_{\psi}(v)| = 1$. Note that no candidate color for $wv$ is valid; in other words, there exists a $(1, \theta, v, w)$-critical path with respect to $\psi$ (the same with respect to $\phi$) for $n+2 \leq \theta \leq \kappa$. Consequently, $\Upsilon(vv_{1}) = \Upsilon(ww_{1}) = \{n+1, \dots, \kappa\}$. Reassigning $3$ to $vv_{1}$ results in another acyclic edge coloring $\phi^{*}$ of $G - vw$. Similarly, we can prove that there exists a $(3, \theta, v, w)$-critical path with respect to $\phi^{*}$ for $n+1 \leq \theta \leq \kappa$, and $\Upsilon(vv_{1}) = \Upsilon(ww_{3}) = \{n+1, \dots, \kappa\}$. Exchanging the colors on $ww_{1}$ and $ww_{3}$, we obtain a new acyclic edge coloring of $G - vw$, but now $n+1$ is valid for $vw$, a contradiction.

\begin{case}%
$\Upsilon(vv_{1}) \supseteq \{n+1, \dots, \kappa\}$ and $\Upsilon(vv_{2}) \supseteq \{n+1, \dots, \kappa\}$.
\end{case}
In fact, $\Upsilon(vv_{1}) = \Upsilon(vv_{2}) = \{n+1, \dots, \kappa\}$. By symmetry, we may assume that $T_{1} \neq \emptyset$. Exchanging the colors on $vv_{1}$ and $vv_{2}$, we obtain a new acyclic edge coloring $\Phi$ of $G - vw$. Note that no candidate color for $wv$ is valid. In other words, if assigning an arbitrary color $\theta_{2}$ in $T_{2}$ to $vw$, then there exists a $(1, \theta_{2})$-dichromatic cycle containing $vw$ with respect to $\Phi$; if assigning an arbitrary color $\theta_{1}$ in $T_{1}$ to $vw$, then there exists a $(2, \theta_{1})$-dichromatic cycle containing $vw$ with respect to $\Phi$. Now, we have $\Upsilon(ww_{1}) = \Upsilon(ww_{2}) = T_{1} \cup T_{2} = \{n+1, \dots, \kappa\}$. Reassigning $3$ to $vv_{1}$ and $1$ to $vv_{2}$, and assigning an arbitrary color $\theta_{1}$ in $T_{1}$ to $vw$, the resulting proper edge coloring has a $(3, \theta_{1})$-dichromatic cycle containing $w_{3}wvv_{1}$. Exchanging the colors on $ww_{1}$ and $ww_{2}$, and reassigning $3$ to $vv_{2}$, we obtain an acyclic edge coloring of $G - vw$. But every color in $T_{1}$ is valid for $vw$ with respect to this acyclic edge coloring of $G - vw$, which derives a contradiction. This completes the proof of Claim~1.
\end{proof}

Without loss of generality, let $\phi(vv_{2}) = n+1$. There exists a $(1, \alpha, v, w)$-critical path for $n+2 \leq \alpha \leq \kappa$, otherwise, the color $\alpha$ is valid for $vw$, a contradiction. Hence, $\mathcal{U}(v_{1}) \supseteq \{1, n+2, \dots, \kappa\}$ and $\mathcal{U}(w_{1}) \supseteq \{1, n+2, \dots, \kappa\}$. Consequently, there exists no $(1, \alpha, v, v_{2})$-critical path for $\alpha \in \{n+2, \dots, \kappa\}$.

\begin{claim}\label{V-W-Path}%
There is a $(1, n+1, w, v_{1})$-alternating path.
\end{claim}
\begin{proof}%
Suppose that there is no $(1, n+1, w, v_{1})$-alternating path. Removing $n+1$ from $vv_{2}$ and reassigning $n+1$ to $wv$, yields an acyclic edge coloring $\sigma$ of $G - vv_{2}$. Fact 2 guarantees $\mathcal{U}_{\sigma}(v) \cap \mathcal{U}_{\sigma}(v_{2}) = \{1\}$. If $\{n+2, \dots, \kappa\} \nsubseteq \Upsilon(vv_{2})$, then every color $\alpha$ in $\{n+2, \dots, \kappa\} \setminus \Upsilon(vv_{2})$ is valid for $vv_{2}$ with respect to $\sigma$ since there is no $(1, \alpha, v, v_{2})$-critical path with respect to $\phi$. It follows that $\{1, n+1, \dots, \kappa\} \subseteq \mathcal{U}(v_{2})$; in fact, we have $\mathcal{U}(v_{2}) = \{1, n+1, \dots, \kappa\}$ and $C(v_{2}) = \{2, \dots, n\}$. Since $|\Upsilon(vv_{1}) \cap \{2, \dots, n\}| \leq 1$, thus we have $\{2, \dots, n\} \setminus \Upsilon(vv_{1}) \neq \emptyset$, but every color in $\{2, \dots, n\} \setminus \Upsilon(vv_{1})$ is valid for $vv_{2}$ with respect to $\sigma$, a contradiction. 
\end{proof}

Now, there exists a $(1, \alpha, w, v_{1})$-alternating path for $n+1 \leq \alpha \leq \kappa$. Hence, $\mathcal{U}(v_{1}) \supseteq \{1, n+1, \dots, \kappa\}$ and $\mathcal{U}(w_{1}) \supseteq \{1, n+1, \dots, \kappa\}$. Furthermore, $\deg(v_{1}) = \deg(w_{1}) =\Delta(G)$ and $\mathcal{U}(v_{1}) = \mathcal{U}(w_{1}) = \{1, n+1, \dots, \kappa\}$. Note that $wv_{1} \notin E(G)$ since $\Upsilon(vw) \cap \Upsilon(vv_{1}) = \emptyset$.

\begin{claim}%
$\{1, 2, \dots, n\} \subseteq \Upsilon(vv_{2})$.
\end{claim}
\begin{proof}%
Suppose, towards a contradiction, that there is a color $\alpha_{2}$ in $\{1, 2, \dots, n\} \setminus \Upsilon(vv_{2})$. Reassigning $\alpha_{2}$ to $vv_{2}$ and reassigning a color $\alpha_{1}$ in $\{1, 2, \dots, n\} \setminus \{\alpha_{2}\}$ to $vv_{1}$, we obtain an acyclic edge coloring $\phi^{*}$ of $G - vw$. But $|\mathcal{U}_{\phi^{*}}(w) \cap \mathcal{U}_{\phi^{*}}(v)| = 2$, which contradicts \autoref{Good-3-vertex}~\ref{3a}.
\end{proof}

Consequently, $\mathcal{U}(v_{2}) \supseteq \{1, 2, \dots, n+1\}$ and $C(v_{2}) \subseteq \{n+2, \dots, \kappa\}$. Without loss of generality, we may assume that $C(v_{2}) = \{n+2, \dots, m\}$. Note that $m \geq 2(\kappa - \Delta(G) + 1)$ since $|C(v_{2})| \geq \kappa - \Delta(G)$.

We may assume that $N_{G}(v_{1}) = \{v, y_{n+1}, y_{n+2}, \dots, y_{\kappa}\}$ with $\phi(v_{1}y_{\lambda}) = \lambda$ for $n+1 \leq \lambda \leq \kappa$, and $N_{G}(v_{2}) = \{v, z_{1}, z_{2}, \dots, z_{n}\} \cup \{z_{m+1}, z_{m+2}, \dots, z_{\kappa}\}$ with $\phi(v_{2}z_{t}) = t$ for $t \in \{1, \dots, n\} \cup \{m+1, \dots, \kappa\}$. 

\begin{claim}\label{thetaPath}%
For $\theta \in \{2, \dots, n\}$ and $\lambda \in \{n+1, \dots, m\}$, there exists a $(\lambda, \theta, v_{1}, v_{2})$-alternating path.
\end{claim}
\begin{proof}%
By symmetry, assume that there exists no $(\lambda, 2, v_{1}, v_{2})$-alternating path. Reassigning $2$ to $vv_{1}$ and $\lambda$ to $vv_{2}$ results in an acyclic edge coloring $\sigma_{1}$ of $G - wv$. By similar arguments as above, we have $\Upsilon(ww_{2}) = \Upsilon(vv_{1}) = \{n+1, \dots, \kappa\}$. Therefore, exchanging the colors on $ww_{1}$ and $ww_{2}$ results in an acyclic edge coloring $\sigma_{2}$ of $G - wv$, but $\kappa$ is valid for $wv$ with respect to this acyclic edge coloring $\sigma_{2}$, which derives a contradiction.
\end{proof}

By \autoref{thetaPath}, we have $\mathcal{U}(y_{\lambda}) \supseteq \{1, \dots, n, \lambda\}$ and $\deg(y_{\lambda}) \geq n+1$ for $\lambda \in \{n+1, \dots, m\}$. Therefore, the vertex $v_{1}$ is adjacent to at least $\kappa - \deg(v_{2}) + 1$ vertices of degree at least $\kappa - \Delta(G) + 2$ and \ref{3d} holds. 

The \autoref{thetaPath} also implies that $\mathcal{U}(z_{t}) \supseteq \{n+1, \dots, m, t\}$ and $\deg(z_{t}) \geq \kappa - \deg(v_{2}) + 2$ for $t \in \{2, \dots, n\}$. Therefore, the vertex $v_{2}$ is adjacent to at least $\kappa - \Delta(G)$ vertices of degree at least $\kappa - \deg(v_{2}) + 2$ and \ref{3e} holds. Since $\Upsilon(ww_{1}) = \{n+1, n+2, \dots, \kappa\}$ and $C(v_{2}) = \{n+2, \dots, m\}$, thus $w_{1} \neq v_{2}$. Since $\Upsilon(v_{2}z_{t}) \supseteq \{n+1, \dots, m\}$ for $t \in \{2, \dots, n\}$ and $\mathcal{U}(w) = \{1, 2, \dots, n\}$, it follows that $w \notin \{z_{2}, \dots, z_{n}\}$, that is, $v_{2} \notin \{w_{2}, \dots, w_{n}\}$. Therefore, the edge $wv$ is not contained in any triangle.

\begin{claim}%
$\deg(v_{2}) \geq \kappa - \Delta(G) + 3$.
\end{claim}
\begin{proof}%
If $\deg(v_{2}) = \kappa - \Delta(G) + 2 = n+1$, then $\mathcal{U}(v_{2}) = \{1, \dots, n+1\}$ and $m = \kappa$. By the above arguments, we have $\Upsilon(v_{2}z_{t}) = \{n+1, \dots, \kappa\}$ for $2 \leq t \leq n$. Exchanging the colors on $v_{2}z_{2}$ and $v_{2}z_{3}$, reassigning $2$ to $vv_{1}$ and $\kappa$ to $wv$, we obtain an acyclic edge coloring of $G$, a contradiction.
\end{proof}

Consequently, $\deg(v_{1}) = \Delta(G) \geq \deg(v_{2}) \geq \kappa - \Delta(G) + 3$ and \ref{3b} holds.

\begin{claim}%
The vertices in $\{w_{2}, w_{3}, \dots, w_{n}\}$ are all $4^{+}$-vertices.
\end{claim}
\begin{proof}
By contradiction and symmetry, suppose that $w_{2}$ is a $3^{-}$-vertex. By \autoref{kappa=2} and \autoref{2++edge}, the vertex $w_{2}$ is a $3$-vertex. Removing the color on $ww_{2}$ and assigning $2$ to $wv$ results in an acyclic edge coloring $\psi$ of $G - ww_{2}$. By \autoref{Good-3-vertex}~\ref{3a}, we have $|\mathcal{U}_{\psi}(w) \cap \mathcal{U}_{\psi}(w_{2})| = 1$. If $\mathcal{U}_{\psi}(w) \cap \mathcal{U}_{\psi}(w_{2}) = \{1\}$, then every color in $\{n+2, \dots, \kappa\} \setminus \Upsilon(ww_{2})$ is valid for $ww_{2}$ with respect to $\psi$, a contradiction. If $\mathcal{U}_{\psi}(w) \cap \mathcal{U}_{\psi}(w_{2}) = \{3\} \subseteq \{3, \dots, n\}$, then we can similarly prove that $\mathcal{U}(w_{3}) = \{3, n+1, \dots, \kappa\}$. Exchanging colors on $ww_{1}$ and $ww_{3}$, we obtain a new acyclic edge coloring of $G - vw$, but $\kappa$ is valid for $vw$ with respect to this coloring, a contradiction. Therefore, the vertices in $\{w_{2}, w_{3}, \dots, w_{n}\}$ are all $4^{+}$-vertices.
\end{proof}

Notice that $w_{1}$ is a vertex with maximum degree, thus $w$ is adjacent to exactly one $3^{-}$-vertex, say $v$.

In what follows, suppose that $v_{2}$ is adjacent to precisely $\kappa - \Delta$ vertices of degree at least four, say $z_{2}, z_{3}, \dots, z_{n}$. By \autoref{2+edge}, the vertices $z_{1}, z_{m+1}, z_{m+2}, \dots, z_{\kappa}$ are all $3$-vertices. Removing the color $1$ from $v_{2}z_{1}$, reassigning $2, 1$ and $n+1$ to $vv_{1}, vv_{2}$ and $vw$ respectively, we obtain an acyclic edge coloring $\pi$ of $G - v_{2}z_{1}$. By Fact 2, we have $\mathcal{U}_{\pi}(v_{2}) \cap \mathcal{U}_{\pi}(z_{1}) \neq \emptyset$.  If $\Upsilon(v_{2}z_{1}) \subseteq \{2, \dots, m\}$, then every color in $\{n+1, \dots, m\} \setminus \Upsilon(v_{2}z_{1})$ is valid for $v_{2}z_{1}$ with respect to $\pi$, a contradiction. So we may assume that $\Upsilon(v_{2}z_{1}) \cap \{m+1, \dots, \kappa\} \neq \emptyset$. 

(1) Suppose that $\Upsilon(v_{2}z_{1}) \subseteq \{m+1, m+2, \dots, \kappa\}$. By symmetry, we may assume that $\Upsilon(v_{2}z_{1}) = \{\kappa-1, \kappa\}$. There exists a $(\kappa-1, \alpha, z_{1}, v_{2})$- or $(\kappa, \alpha, z_{1}, v_{2})$-critical path for $\alpha \in \{n+1, \dots, m\}$. Since $|\{n+1, \dots, m\}| \geq 3$ and $\deg(z_{\kappa}) = \deg(z_{\kappa-1}) = 3$, so we may assume that there exists a $(\kappa, \alpha_{1}, z_{1}, v_{2})$- and a $(\kappa, \alpha_{2}, z_{1}, v_{2})$-critical path, where $\{\alpha_{1}, \alpha_{2}\} \subseteq \{n+1, \dots, m\}$. Thus, we have $\Upsilon(v_{2}z_{\kappa}) = \{\alpha_{1}, \alpha_{2}\}$. Modify $\pi$ by reassigning $\alpha_{3}$ to $v_{2}z_{\kappa}$, where $\alpha_{3} \in \{n+1, \dots, m\} \setminus \{\alpha_{1}, \alpha_{2}\}$, and we obtain a new acyclic edge coloring $\varphi$ of $G - v_{2}z_{1}$. Neither $\alpha_{1}$ nor $\alpha_{2}$ is valid for $v_{2}z_{1}$ with respect to $\varphi$, and it follows that $\Upsilon(v_{2}z_{\kappa-1}) = \{\alpha_{1}, \alpha_{2}\}$. Now, we have $\Upsilon(v_{2}z_{\kappa}) = \Upsilon(v_{2}z_{\kappa-1}) = \{\alpha_{1}, \alpha_{2}\}$, but $\alpha_{3}$ is valid for $v_{2}z_{1}$ with respect to $\pi$, a contradiction. 

(2) Suppose that $\Upsilon(v_{2}z_{1}) = \{s, m+1\}$ with $s \in \{2, \dots, m\}$. Since no candidate color for $v_{2}z_{1}$ is valid with respect to $\pi$, it follows that there exists a $(m+1, \theta, z_{1}, v_{2})$-critical path with respect to $\pi$ for $\theta \in \{n+1, \dots, m\} \setminus \{s\}$, thus $\deg(v_{2}) = \Delta(G) = \kappa - 2$. Furthermore, we have $n = 3$, $m = 6$ and $\Upsilon(v_{2}z_{7}) \cup \{s\} = \{4, 5, 6\}$. If there exists no $(2, 7, v_{2}, v_{1})$-alternating path, then assigning $4, 2, 7$ and $s$ to $vw, vv_{1}, vv_{2}$ and $v_{2}z_{7}$, we obtain an acyclic edge coloring of $G$, a contradiction. Thus, there exists a $(2, 7, v_{2}, v_{1})$-alternating path and $7 \in \Upsilon(v_{2}z_{2})$. Similarly, there exists a $(3, 7, v_{2}, v_{1})$-alternating path and $7 \in \Upsilon(v_{2}z_{3})$. By \autoref{thetaPath}, we have $\Upsilon(v_{2}z_{2}) \cap \Upsilon(v_{2}z_{3}) \supseteq \{4, 5, 6, 7\}$. If $\Delta(G) = 5$, then $\Upsilon(v_{2}z_{2}) = \Upsilon(v_{2}z_{3}) = \{4, 5, 6, 7\}$, thus exchanging the colors on $v_{2}z_{2}$ and $v_{2}z_{3}$, and reassigning $5$ to $vw$ and $2$ to $vv_{1}$, brings us an acyclic edge coloring of $G$, a contradiction. So we may assume that $\Delta(G) \geq 6$ and $\kappa = \Delta(G) + 2 \geq 8$.

If $\Upsilon(v_{2}z_{8}) \subseteq \{1, \dots, 6\}$, then modify $\pi$ by reassigning $8, 7$ and a color in $\Upsilon(v_{2}z_{7})$ to $v_{2}z_{7}, v_{2}z_{8}$ and $v_{1}z_{1}$. We obtain an acyclic edge coloring of $G$, a contradiction. Thus, we have $|\Upsilon(v_{2}z_{8}) \cap \{1, \dots, 6\}| \leq 1$. By symmetry, we may assume that $|\Upsilon(v_{2}z_{t}) \cap \{1, \dots, 6\}| \leq 1$ for $t \in \{8, \dots, \kappa\}$.

Suppose that $\Upsilon(v_{2}z_{p}) = \{7, q\}$ for some $p \geq 8$. We can modify $\pi$ by removing $p$ from $v_{2}z_{p}$ and reassigning $p$ to $v_{2}z_{1}$, and then we obtain an acyclic edge coloring $\pi_{2}$ of $G - v_{2}z_{p}$. Since there exists no $(7, \theta, z_{p}, v_{2})$-critical path with respect to $\pi_{2}$ for $\theta \in \{4, 5, 6\} \setminus \{q\}$, thus there exists a $(q, \theta, z_{p}, v_{2})$-critical path with respect to $\pi_{2}$.  Hence, $q \geq 8$ and $\{4, 5, 6\} \subseteq \Upsilon(v_{2}z_{q})$, which contradicts the fact that $z_{q}$ is a $3$-vertex. Therefore, $7 \notin \Upsilon(v_{2}z_{t})$ for $t \geq 8$. We can modify $\pi$ by exchanging the colors on $v_{2}z_{7}$ and $v_{2}z_{8}$, and reassigning a color in $\Upsilon(v_{2}z_{7})$ to $v_{2}z_{1}$, but this yields an acyclic edge coloring of $G$.
\resetcounter
\end{proof}

\begin{lemma}[Hou \etal \cite{MR2849391}]\label{3+vertex}%
Let $G$ be a $\kappa$-deletion-minimal graph with $\kappa \geq \Delta(G) + 2$. If $v$ is a $3$-vertex, then every neighbor of $v$ is a $(\kappa - \Delta(G) + 2)^{+}$-vertex.
\end{lemma}
A $3$-vertex is a {\em special} $3$-vertex if it is adjacent to a $(\kappa - \Delta(G) + 2)$-vertex, otherwise, it is called a {\em normal} $3$-vertex. In other words, a vertex is a normal $3$-vertex if it is a $3$-vertex and every neighbor of $v$ is a $(\kappa - \Delta(G) + 3)^{+}$-vertex by \autoref{3+vertex}.

\begin{lemma}\label{N_3_N}%
If $G$ is a $\kappa$-deletion-minimal graph with $\kappa \geq \Delta(G) + 2$ and $w$ is a $(\kappa - \Delta(G) + 3)$-vertex, then $w$ is adjacent to at most $\kappa - \Delta(G) + 1$ vertices of degree three.
\end{lemma}
\begin{proof}%
Let $N_{G}(w) = \{w_{0}, w_{1}, \dots, w_{\tau-1}\}$, where $\tau = \kappa - \Delta(G) + 3$. To derive a contradiction, assume that $w$ is adjacent to at least $\tau-1$ vertices of degree three and $w_{0}$ is a $3$-vertex with $N_{G}(w_{0}) = \{w, v_{1}, v_{2}\}$. Since $G$ is $\kappa$-deletion-minimal, it follows that $G - ww_{0}$ admits an acyclic edge coloring $\phi$ with $\phi(ww_{i}) = i$ for $1 \leq i \leq \tau-1$. Since $\deg(w)+ \deg(w_{0}) = \deg(w) + 3 \neq \kappa + 2$, Fact 2 guarantees $|\mathcal{U}(w) \cap \mathcal{U}(w_{0})| \geq 1$. Without loss of generality, assume that $\phi(w_{0}v_{1}) = 1$.

\begin{case}%
$|\mathcal{U}(w) \cap \mathcal{U}(w_{0})| = 1$.
\end{case}

By symmetry, we assume that $\phi(w_{0}v_{2}) = \tau$. There exists a $(1, \alpha, w_{0}, w)$-critical path for $\alpha \in \{\tau+1, \dots, \kappa\}$, thus $\{\tau+1, \dots, \kappa\} \subseteq \Upsilon(ww_{1}) \cap \Upsilon(w_{0}v_{1})$. Consequently, there exists no $(1, \alpha, w_{0}, v_{2})$-critical path for $\alpha \in \{\tau+1, \dots, \kappa\}$.

\begin{claim}\label{v-w-path}%
There exists a $(1, \tau, w, v_{1})$-alternating path.
\end{claim}
\begin{proof}%

Suppose that there exists no $(1, \tau, w, v_{1})$-alternating path. Removing $\tau$ from $w_{0}v_{2}$ and reassigning $\tau$ to $ww_{0}$, yields an acyclic edge coloring $\phi^{*}$ of $G - w_{0}v_{2}$. Fact 2 guarantees $\mathcal{U}_{\phi^{*}}(w_{0}) \cap \mathcal{U}_{\phi^{*}}(v_{2}) = \{1\}$. If $\{\tau+1, \dots, \kappa\} \nsubseteq \Upsilon(w_{0}v_{2})$, then every color $\alpha$ in $\{\tau+1, \dots, \kappa\} \setminus \Upsilon(w_{0}v_{2})$ is valid for $w_{0}v_{2}$ with respect to $\phi^{*}$ since there is no $(1, \alpha, w_{0}, v_{2})$-critical path with respect to $\phi$. It follows that $\mathcal{U}(v_{2}) \supseteq \{1, \tau, \dots, \kappa\}$ and then $C(v_{2}) \subseteq \{2, \dots, \tau-1\}$. If there exists $\beta \in C(v_{1}) \cap C(v_{2})$, then $\beta$ is valid for $w_{0}v_{2}$ with respect to $\phi^{*}$, which is a contradiction. Thus, we have $C(v_{2}) \subseteq \mathcal{U}(v_{1})$ and $C(v_{1}) \subseteq \mathcal{U}(v_{2})$. Note that $C(v_{2}) \subseteq \{2, \dots, \tau-1\}$ and $|C(v_{2})| \geq \kappa - \Delta(G)$. It follows that $\deg(v_{1}) = |\{\tau+1, \dots, \kappa\}| + |\mathcal{U}(v_{1}) \cap \{1, \dots, \tau\}| \geq \Delta(G) - 3 + (1 + |C(v_{2})|) \geq \kappa - 2 \geq \Delta(G)$. Hence, $\deg(v_{1}) = \deg(v_{2})= \Delta(G)$ and $\kappa = \Delta(G) + 2$, and then $\tau = 5$ and $|C(v_{2})| = 2$. Without loss of generality, assume that $C(v_{2}) = \{2, 3\}$, and then $\mathcal{U}(v_{1}) = \{1, 2, 3, 6, \dots, \kappa\}$ and $\mathcal{U}(v_{2}) = \{1, 4, 5, 6, \dots, \kappa\}$. Reassigning $5$ to $w_{0}v_{1}$ and $2$ to $w_{0}v_{2}$ results in an acyclic edge coloring $\phi_{2}$ of $G - ww_{0}$, thus there exists a $(2, \alpha, w_{0}, w)$-critical path with respect to $\phi_{2}$ for $6 \leq \alpha \leq \kappa$, and then $\Upsilon(w_{0}w_{2}) \supseteq \{6, \dots, \kappa\}$. Similarly, reassigning $5$ to $w_{0}v_{1}$ and $3$ to $w_{0}v_{2}$ results in an acyclic edge coloring $\phi_{3}$ of $G - ww_{0}$, thus there exists a $(3, \alpha, w_{0}, w)$-critical path with respect to $\phi_{3}$ for $6 \leq \alpha \leq \kappa$, and then $\Upsilon(w_{0}w_{3}) \supseteq \{6, \dots, \kappa\}$. Reassigning $4$ to $w_{0}v_{1}$ results in an acyclic edge coloring $\phi_{4}$ of $G - ww_{0}$, thus there exists a $(4, \alpha, w_{0}, w)$-critical path with respect to $\phi_{4}$ for $6 \leq \alpha \leq \kappa$, and then $\Upsilon(w_{0}w_{4}) \supseteq \{6, \dots, \kappa\}$. That is, $\Upsilon(ww_{i}) \supseteq \{6, \dots, \kappa\}$ for $1 \leq i \leq 4$. Recall that $w$ is adjacent to at least four $3$-vertices, including $w_{i}$ and $w_{j}$ ($1 \leq i < j \leq 4$). Hence, $\kappa = 7$ and $\Upsilon(ww_{i}) = \Upsilon(ww_{j}) = \{6, 7\}$. Modify $\phi_{j}$ by exchanging the colors on $ww_{i}$ and $ww_{j}$, and assigning $6$ to $ww_{0}$, we obtain an acyclic edge coloring of $G$, a contradiction.
\end{proof}

Hence, $\mathcal{U}(v_{1}) \cap \mathcal{U}(w_{1}) \supseteq \{1, \tau, \dots, \kappa\}$. The degree of $w_{1}$ is at least four since $|\{1, \tau, \dots, \kappa\}| \geq \Delta(G) - 1 \geq 4$. It follows that $w_{2}, \dots, w_{\tau-1}$ are all $3$-vertices.

Note that $C(w_{1}) \subseteq \{2, \dots, \tau-1\}$. Without loss of generality, assume that $\{2, 3, \dots, \tau-2\} \subseteq C(w_{1}) \subseteq \{2, 3, \dots, \tau-1\}$. Removing the color on $ww_{2}$ and assigning $2$ to $ww_{0}$ results in an acyclic edge coloring $\psi$ of $G - ww_{2}$. Suppose that $\mathcal{U}_{\psi}(w_{2}) \cap \mathcal{U}_{\psi}(w) = \{\lambda\}$. By the above arguments, we have that $w_{\lambda}$ is a $4^{+}$-vertex, and then $\lambda = 1$. We can extend $\psi$ by assigning a color in $\{\tau, \dots, \kappa\} \setminus \mathcal{U}_{\psi}(w_{2})$ to obtain an acyclic edge coloring of $G$, which is a contradiction. Hence, $|\mathcal{U}_{\psi}(w_{2}) \cap \mathcal{U}_{\psi}(w)| = 2$ and $|\mathcal{U}(w_{2}) \cap \{1, \dots, \tau-1\}| = 3$. Similarly, we can prove that $|\mathcal{U}(w_{i}) \cap \{1, \dots, \tau-1\}| = 3$ for $2 \leq i \leq \tau-2$.

The candidate color $\kappa$ for $ww_{2}$ is not valid with respect to $\psi$, thus $\tau-1 \in \Upsilon(ww_{2})$ and there exists a $(\tau-1, \kappa, w, w_{2})$-critical path with respect to $\psi$. Removing the color $3$ on $ww_{3}$ and assigning $3$ and $\kappa$ to $ww_{0}$ and $ww_{3}$ results in an acyclic edge coloring of $G$, which is a contradiction.

\begin{case}%
$|\mathcal{U}(w) \cap \mathcal{U}(w_{0})| = 2$.
\end{case}

By symmetry, we may assume that $\phi(w_{0}v_{2}) = 2$. There exists a $(1, \alpha, w_{0}, w)$-critical path or $(2, \alpha, w_{0}, w)$-critical path for $\alpha \in \{\tau, \dots, \kappa\}$, thus $\{\tau, \dots, \kappa\} \subseteq \Upsilon(ww_{1}) \cup \Upsilon(ww_{2})$.

\begin{subcase}%
Either $\Upsilon(w_{0}v_{1}) \nsupseteq \{\tau, \dots, \kappa\}$ or $\Upsilon(w_{0}v_{2}) \nsupseteq \{\tau, \dots, \kappa\}$.
\end{subcase}
By symmetry, assume that $\tau \notin \Upsilon(w_{0}v_{2})$. Note that $\tau$ is not valid for $ww_{0}$, it follows that there exists a $(1, \tau, w_{0}, w)$-critical path, and then there exists no $(1, \tau, w_{0}, v_{2})$-critical path. Reassigning $\tau$ to $w_{0}v_{2}$ results in a new acyclic edge coloring $\sigma$ of $G - ww_{0}$ with $|\mathcal{U}_{\sigma}(w) \cap \mathcal{U}_{\sigma}(w_{0})| = 1$ and it takes us back to Case 1.

\begin{subcase}%
$\Upsilon(w_{0}v_{1}) \supseteq \{\tau, \dots, \kappa\}$ and $\Upsilon(w_{0}v_{2}) \supseteq \{\tau, \dots, \kappa\}$.
\end{subcase}

Suppose that $w_{1}$ and $w_{2}$ are all $3$-vertices. Since $|\{\tau, \dots, \kappa\}| \geq 3$, we may assume that $\Upsilon(ww_{2}) \subseteq \{\tau, \dots, \kappa\}$. But reassigning a color in $\{\tau, \dots, \kappa\} \setminus \Upsilon(ww_{2})$ to $ww_{2}$ results in a new acyclic edge coloring  of $G - ww_{0}$ and it takes us back to Case 1.

So we may assume that $w_{1}$ is a $4^{+}$-vertex. Note that $C(v_{1}) \subseteq \{2, 3, \dots, \tau-1\}$ and $C(v_{2}) \subseteq \{1, 3, 4, \dots, \tau-1\}$, it follows that  $C(v_{1}) \cap \{3, \dots, \tau-1\} \neq \emptyset$, say $3 \in C(v_{1})$. Reassigning $3$ to $w_{0}v_{1}$ creates a $(3, 2)$-dichromatic cycle containing $w_{0}v_{1}$, otherwise, by the above arguments, one of $w_{2}$ and $w_{3}$ must be a $4^{+}$-vertex, a contradiction. Hence, $2 \in \Upsilon(w_{0}v_{1})$ and $3 \in \Upsilon(w_{0}v_{2})$. Hence, $C(v_{1}) = \{3, 4, \dots, \tau-1\}$ and $C(v_{2}) = \{1, 4, 5, \dots, \tau-1\}$. Reassigning $4$ to $w_{0}v_{1}$ results in another acyclic edge coloring of $G - ww_{0}$. We can similarly prove that one of $w_{2}$ and $w_{4}$ is a $4^{+}$-vertex, a contradiction.
\resetcounter
\end{proof}

\begin{lemma}\label{L9}%
Let $G$ be a $\kappa$-deletion-minimal graph with $\kappa \geq \Delta(G) + 2$, and let $w_{0}$ be a $3$-vertex with $N_{G}(w_{0}) = \{w, w_{1}, w_{2}\}$, and $\deg(w) = \kappa - \Delta(G) + 3$. If $ww_{1}, ww_{2} \in E(G)$, then $\deg(w_{1}) = \deg(w_{2}) = \Delta(G)$ and $w$ is adjacent to precisely one vertex (namely $w_{0}$) of degree less than $\Delta(G) - 1$.
\end{lemma}

\begin{proof}
Let $N_{G}(w) = \{w_{0}, w_{1}, \dots, w_{n}\}$, where $n = \kappa - \Delta(G) + 2 \geq 4$. Since $G$ is a $\kappa$-deletion-minimal graph, it follows that $G - ww_{0}$ admits an acyclic edge coloring $\phi$ with $\phi(ww_{i}) = i$ for $1 \leq i \leq n$. Since $\deg(w)+ \deg(w_{0}) =\deg(w) + 3 \neq \kappa + 2$, Fact 2 guarantees $|\mathcal{U}(w) \cap \mathcal{U}(w_{0})| \geq 1$. By symmetry, we may assume that the color $\lambda$ on $w_{0}w_{2}$ is a common color.

\begin{case}\label{L91}%
$|\mathcal{U}(w) \cap \mathcal{U}(w_{0})| = 1$.
\end{case}
Without loss of generality, we may assume that $\phi(w_{0}w_{1}) = n+1$. There exists a $(\lambda, \alpha, w_{0}, w)$-critical path for $\alpha \in \{n+2, \dots, \kappa\}$, thus $\{n+2, \dots, \kappa\} \subseteq \Upsilon(w_{0}w_{2}) \cap \Upsilon(ww_{\lambda})$. Note that $|\{n+2, \dots, \kappa\}| = \Delta(G) - 3$.
\begin{subcase}\label{L911}%
The color on $ww_{1}$ is the common color.
\end{subcase}
It follows that $\lambda = 1$ and $\{1\} \cup \{n+2, \dots, \kappa\} \subseteq \Upsilon(w_{0}w_{1})$. Removing $n+1$ from $w_{0}w_{1}$ and assigning $n+1$ to $ww_{0}$, we obtain an acyclic edge coloring $\psi$ of $G - w_{0}w_{1}$. If $\{3, 4, \dots, n\} \setminus(\Upsilon(w_{0}w_{1}) \cup \Upsilon(w_{0}w_{2})) \neq \emptyset$, then every color in $\{3, 4, \dots, n\} \setminus(\Upsilon(w_{0}w_{1}) \cup \Upsilon(w_{0}w_{2}))$ is valid for $w_{0}w_{1}$ with respect to $\psi$, a contradiction. Thus, we have $\{3, 4, \dots, n\} \setminus(\Upsilon(w_{0}w_{1}) \cup \Upsilon(w_{0}w_{2})) = \emptyset$. Notice that $|\{3, 4, \dots, n\} \cap \Upsilon(w_{0}w_{1})| \leq 1$ and $|\{3, 4, \dots, n\} \cap \Upsilon(w_{0}w_{2})| \leq 1$, thus $n = 4$, $|\{3, 4\} \cap \Upsilon(w_{0}w_{1})| = 1$ and $|\{3, 4\} \cap \Upsilon(w_{0}w_{2})| = 1$. We may assume that $\mathcal{U}(w_{1}) = \{1, \dots, \kappa\} \setminus \{2, 3\}$ and $\mathcal{U}(w_{2}) = \{1, \dots, \kappa\} \setminus \{4, 5\}$. Clearly, $\deg(w_{1}) = \deg(w_{2}) = \Delta(G)$. 

Reassigning $4$ to $w_{0}w_{2}$ yields a new acyclic edge coloring $\sigma_{1}$ of $G - ww_{0}$, thus $\{6, \dots, \kappa\} \subseteq \Upsilon(ww_{4})$. Reassigning $3$ to $w_{0}w_{1}$ and $5$ to $w_{0}w_{2}$ yields another acyclic edge coloring $\sigma_{2}$ of $G - ww_{0}$, thus $\{6, \dots, \kappa\} \subseteq \Upsilon(ww_{3})$. If $2 \notin \Upsilon(ww_{4})$, then reassigning $5, 2$ and $4$ to $w_{0}w, w_{0}w_{1}$ and $w_{0}w_{2}$ yields an acyclic edge coloring of $G$. Thus, we have $\{2\} \cup \{6, 7, \dots, \kappa\} \subseteq \Upsilon(ww_{4})$ and $\deg(w_{4}) \geq \Delta(G) - 1$. If $1 \notin \Upsilon(ww_{3})$, then reassigning $5$ to $w_{0}w$ and $3$ to $w_{0}w_{1}$ yields an acyclic edge coloring of $G$. Thus, we have $\{1\} \cup \{6, 7, \dots, \kappa\} \subseteq \Upsilon(ww_{3})$ and $\deg(w_{3}) \geq \Delta(G) - 1$.
\begin{subcase}\label{L912}%
The color on $ww_{1}$ is not a common color.
\end{subcase}

If $1 \notin \Upsilon(w_{0}w_{2})$, then reassigning $1$ to $w_{0}w_{2}$ will take us back to Subcase~\ref{L911}. So we may assume that $1 \in \Upsilon(w_{0}w_{2})$. Hence, $\deg(w_{2}) = \Delta(G)$ and $\Upsilon(w_{0}w_{2}) = \{n+2, \dots, \kappa\} \cup \{1, 2\}$. Removing $n+1$ from $w_{0}w_{1}$ and assigning $n+1$ to $ww_{0}$, we obtain an acyclic edge coloring $\varphi$ of $G - w_{0}w_{1}$. If $\lambda \notin \Upsilon(w_{0}w_{1})$, then every color in $C(w_{1}) \setminus \{\lambda\}$ is valid for $w_{0}w_{1}$ with respect to $\varphi$, a contradiction. If $\lambda \in \Upsilon(w_{0}w_{1})$, then every color in $C(w_{1}) \setminus \{2\}$ is valid for $w_{0}w_{1}$ with respect to $\varphi$, a contradiction.

\begin{case}%
$|\mathcal{U}(w) \cap \mathcal{U}(w_{0})| = 2$.
\end{case}
Suppose that the color on $w_{0}w_{1}$ is $\lambda^{*}$. There exists a $(\lambda^{*}, \alpha, w_{0}, w)$-critical path or $(\lambda, \alpha, w_{0}, w)$-critical path for $n+1 \leq \alpha \leq \kappa$. Let
\begin{align*}%
T_{1} = \{\alpha_{1} \mid \alpha_{1} \in \{n+1, \dots, \kappa\} \mbox{ and there exists a $(\lambda^{*}, \alpha_{1}, v, w)$-critical path}\},\\
T_{2} = \{\alpha_{2} \mid \alpha_{2} \in \{n+1, \dots, \kappa\} \mbox{ and there exists a $(\lambda, \alpha_{2}, v, w)$-critical path}\}.
\end{align*}
Hence, $T_{1} \cup T_{2} = \{n+1, \dots, \kappa\}$.

Suppose that either $\Upsilon(w_{0}w_{1}) \nsupseteq \{n+1, \dots, \kappa\}$ or $\Upsilon(w_{0}w_{2}) \nsupseteq \{n+1, \dots, \kappa\}$. By symmetry, assume that $n+1 \notin \Upsilon(w_{0}w_{1})$. By the assumption, it follows that there exists a $(\lambda, n+1, w_{0}, w)$-critical path. Reassigning $n+1$ to $w_{0}w_{1}$ results in a new acyclic edge coloring $\varrho$ of $G - ww_{0}$ with $|\mathcal{U}_{\varrho}(w) \cap \mathcal{U}_{\varrho}(w_{0})| = 1$, and this takes us back to Case~\ref{L91}.

So we may assume that $\Upsilon(w_{0}w_{1}) \supseteq \{n+1, \dots, \kappa\}$ and $\Upsilon(w_{0}w_{2}) \supseteq \{n+1, \dots, \kappa\}$. In fact, $\Upsilon(w_{0}w_{1}) = \{n+1, \dots, \kappa\} \cup \{1\}$, $\Upsilon(w_{0}w_{2}) = \{n+1, \dots, \kappa\} \cup \{2\}$ and $\deg(w_{1}) = \deg(w_{2}) = \Delta(G)$.

Since $\{3, 4\} \cap (\Upsilon(w_{0}w_{1}) \cup \Upsilon(w_{0}w_{2})) = \emptyset$, so we may assume that $\lambda^{*} = 3$ and $\lambda = 4$. Reassigning $3$ to $w_{0}w_{2}$ and reassigning an arbitrary color $\lambda_{0}$ in $\{4, \dots, n\}$ to $w_{0}w_{1}$, we obtain a new acyclic edge coloring $\Phi_{1}$ of $G - ww_{0}$. Note that no candidate color for $ww_{0}$ is valid. In other words, if assigning an arbitrary color $\alpha_{1}$ in $T_{1}$ to $ww_{0}$, then there exists a $(\lambda_{0}, \alpha_{1})$-dichromatic cycle containing $ww_{0}$, thus $T_{1} \subseteq \Upsilon(ww_{i})$ for $3 \leq i \leq n$. Similarly, reassigning $4$ to $w_{0}w_{1}$ and reassigning an arbitrary color $\lambda_{1}$ in $\{3\} \cup \{5, \dots, n\}$ to $w_{0}w_{2}$, we obtain another acyclic edge coloring $\Phi_{2}$ of $G - ww_{0}$, and we can prove that $T_{2} \subseteq \Upsilon(ww_{i})$ for $3 \leq i \leq n$. Now, we have $\Upsilon(ww_{i}) \supseteq T_{1} \cup T_{2} = \{n+1, \dots, \kappa\}$ for $3 \leq i \leq n$, thus $\deg(w_{i}) \geq \Delta(G) -1$ for $3 \leq i \leq n$.
\resetcounter
\end{proof}

\begin{lemma}\label{NO44t}%
Let $G$ be a $\kappa$-deletion-minimal graph with $\kappa \geq \Delta(G) + 3$. If a $\tau$-vertex $w$ is contained in a $(4, 4, \tau)$-triangle, then $w$ is adjacent to at most $(\tau - 3)$ vertices of degree three.
\end{lemma}
\begin{proof}
Let $w$ be a $\tau$-vertex with neighborhood $\{w_{0}, w_{1}, \dots, w_{\tau-1}\}$. Suppose that $ww_{1}w_{2}$ is a triangle with $\deg(w_{1}) = \deg(w_{2}) = 4$ and $w_{0}, w_{3}, \dots, w_{\tau-1}$ are $3$-vertices. The graph $G - ww_{0}$ admits an acyclic edge coloring $\phi$ with $\phi(ww_{i}) = i$ for $1 \leq i \leq \tau-1$. Since $\deg(w) + \deg(w_{0}) = \deg(w) + 3 \neq \kappa + 2$, Fact 2 guarantees $|\mathcal{U}(w) \cap \mathcal{U}(w_{0})| \geq 1$.

\begin{case}\label{LC1}%
$|\mathcal{U}(w) \cap \mathcal{U}(w_{0})| = 1$.
\end{case}

\begin{subcase}%
By symmetry, assume that $\Upsilon(ww_{0}) = \{3, \tau\}$.
\end{subcase}
It follows that there exists a $(3, \alpha, w, w_{0})$-critical path for $\alpha \in \{\tau+1, \tau+2, \dots, \kappa\}$, thus $\{\tau+1, \tau+2, \dots, \kappa\} \subseteq \Upsilon(ww_{3})$ and $\deg(w_{3}) \geq 1 + |\{\tau+1, \tau+2, \dots, \kappa\}| \geq 4$, which contradicts the fact that $\deg(w_{3}) = 3$.

\begin{subcase}%
By symmetry, assume that $\Upsilon(ww_{0}) = \{1, \tau\}$.
\end{subcase}
It follows that there exists a $(1, \alpha, w, w_{0})$-critical path for $\alpha \in \{\tau+1, \tau+2, \dots, \kappa\}$, thus $\{\tau+1, \tau+2, \dots, \kappa\} \subseteq \Upsilon(ww_{1})$. Moreover, $\kappa = \tau+3 = \Delta(G) + 3$ and $1 \in \Upsilon(ww_{2})$. By symmetry, we may assume that the color on $w_{1}w_{2}$ is $\tau+1$. Reassigning $\tau$ to $ww_{1}$ results in a new acyclic edge coloring $\psi$ of $G - ww_{0}$. There exists a $(\tau, \tau+1, w, w_{0})$-critical path with respect to $\psi$, thus we have $\tau \in \Upsilon(ww_{2})$. Notice that $\Upsilon(ww_{2}) = \{1, \tau, \tau+1\}$ and $\Upsilon(ww_{1}) = \{\tau+1, \tau+2, \tau+3\}$. Reassigning $\tau+2$ to $ww_{2}$ and $2$ to $ww_{0}$ results in an acyclic edge coloring of $G$.
\begin{case}%
$\Upsilon(ww_{0}) \subseteq \{1, 2, \dots, \tau-1\}$.
\end{case}
\begin{subcase}\label{LC21}%
$\Upsilon(ww_{0}) \subseteq \{3, 4, \dots, \tau-1\}$
\end{subcase}
By symmetry, assume that $\Upsilon(ww_{0}) = \{3, 4\}$. There exists a $(3, \alpha, w, w_{0})$-critical path or $(4, \alpha, w, w_{0})$-critical path for $\alpha \in \{\tau, \tau+1, \dots, \kappa\}$. It follows that $\{\tau, \tau+1, \dots, \kappa\} \subseteq \Upsilon(ww_{3}) \cup \Upsilon(ww_{4})$. Note that $|\{\tau, \tau+1, \dots, \kappa\}| \geq 4$ and $|\Upsilon(ww_{3}) \cup \Upsilon(ww_{4})| \leq 4$, it follows that $\kappa = \tau + 3 = \Delta(G) + 3$ and $\Upsilon(ww_{3}) \cap \Upsilon(ww_{4}) = \emptyset$. By symmetry, assume that $\Upsilon(ww_{3}) = \{\tau, \tau+1\}$ and $\Upsilon(ww_{4}) = \{\tau+2, \tau+3\}$. But reassigning $\tau+1$ to $ww_{4}$ will take us back to Case~\ref{LC1}.
\begin{subcase}\label{LC22}%
$\Upsilon(ww_{0}) = \{1, 2\}$.
\end{subcase}
It follows that there exists a $(1, \alpha, w, w_{0})$-critical path or $(2, \alpha, w, w_{0})$-critical path for $\alpha \in \{\tau, \tau+1, \dots, \kappa\}$, thus $\{\tau, \tau+1, \dots, \kappa\} \subseteq \Upsilon(ww_{1}) \cup \Upsilon(ww_{2})$. Let $\alpha_{0}$ be the color on $w_{1}w_{2}$. If there exists a $(2, \alpha_{0}, w, w_{0})$-critical path passing through $w_{1}w_{2}$, then $2 \in \Upsilon(ww_{1})$ and $\Upsilon(ww_{1}) \cup \Upsilon(ww_{2}) = \{\tau, \tau+1, \tau+2, \tau+3\} \cup \{2\}$, but exchanging the colors on $w_{1}w_{2}$ and $w_{1}w$ will take us back to Case~\ref{LC1}. Similarly, if there exists a $(1, \alpha_{0}, w, w_{0})$-critical path passing through $w_{1}w_{2}$, we can also reduce the proof to Case~\ref{LC1}. Now, we conclude that $\{\tau, \tau+1, \dots, \kappa\} \subseteq (\Upsilon(ww_{1}) \cup \Upsilon(ww_{2})) \setminus \{\alpha_{0}\}$, thus $|\{\tau, \tau+1, \dots, \kappa\}| = 4$ and $\Upsilon(ww_{1}) \cap \Upsilon(ww_{2}) = \{\alpha_{0}\}$. By symmetry, assume that $\Upsilon(ww_{1}) = \{3, \tau, \tau+1\}$ and $\Upsilon(ww_{2}) = \{3, \tau+2, \tau+3\}$. If there exists no $(3, \tau, w_{2}, w)$-critical path, then reassigning $\tau$ to $ww_{2}$ will take us back to Case~\ref{LC1}. Thus, there exists a $(3, \tau, w_{2}, w)$-critical path and $\tau \in \Upsilon(ww_{3})$. Similarly, there exists a $(3, \tau+1, w_{2}, w)$-critical path and $\tau+1 \in \Upsilon(ww_{3})$. Thus, we have $\Upsilon(ww_{3}) = \{\tau, \tau+1\}$. But we can assign $\tau+2$ to $ww_{1}$ and go back to Case~\ref{LC1}.

\begin{subcase}%
$|\Upsilon(ww_{0}) \cap \{1, 2\}| = 1$ and $|\Upsilon(ww_{0}) \cap \{3, 4, \dots, \tau-1\}| = 1$. 
\end{subcase}
By symmetry, assume that $\Upsilon(ww_{0}) = \{1, 3\}$. It follows that $\{\tau, \tau+1, \dots, \kappa\} \subseteq \Upsilon(ww_{1}) \cup \Upsilon(ww_{3})$. Since $|\{\tau, \tau+1, \dots, \kappa\}| \geq 4$ and $|\Upsilon(ww_{1})| = 3$, it follows that $\{\tau, \tau+1, \dots, \kappa\} \setminus \Upsilon(ww_{1}) \neq \emptyset$. By symmetry, we may assume that $\tau \notin \Upsilon(ww_{1})$. Thus there exists a $(3, \tau, w, w_{0})$-critical path and $\tau \in \Upsilon(ww_{3})$. If $\Upsilon(ww_{1}) \subseteq \{\tau, \tau+1, \dots, \kappa\}$, then reassigning $\tau$ to $ww_{1}$ will take us back to Case~\ref{LC1}. So we may assume that $\Upsilon(ww_{1}) \nsubseteq \{\tau, \tau+1, \dots, \kappa\}$ and $\{\tau, \tau+1\} \cap \Upsilon(ww_{1}) = \emptyset$. There exists a $(3, \tau+1, w, w_{0})$-critical path and $\Upsilon(ww_{3}) = \{\tau, \tau+1\}$. But now, reassigning $\tau+2$ to $ww_{3}$ will take us back to Case~\ref{LC1} again.
\resetcounter
\end{proof}

\section{Applications}\label{MResults}

\begin{theorem}[Basavaraju and Chandran \cite{MR2527638}]\label{Non4Regular}%
If $G$ is a connected graph with $\Delta(G) \leq 4$ and $G$ is not $4$-regular, then $\chiup_{a}'(G) \leq \Delta(G) + 2$.
\end{theorem}
\begin{proof}%
Let $G$ be a counterexample with fewest edges and fix $\kappa = \Delta(G) + 2$. For every proper subgraph $H$ of $G$, every component of $H$ is connected with maximum degree at most four and is not $4$-regular, thus every component of $H$ (hence $H$) has an acyclic edge coloring with at most $\kappa$ colors, which implies that $G$ is a $\kappa$-deletion-minimal graph. By \autoref{kappa=2}, the graph $G$ is $2$-connected and $\delta(G) \geq 2$. By \autoref{2++edge}, \autoref{Good-3-vertex} and \autoref{3+vertex}, the graph $G$ contains neither $2$-vertices nor $3$-vertices, thus $G$ is $4$-regular.
\end{proof}

\begin{theorem}[Basavaraju and Chandran \cite{MR2466974}]%
If $G$ is a connected subcubic graph which is not $3$-regular, then $\chiup_{a}'(G) \leq \Delta(G) + 1$.
\end{theorem}
\begin{proof}%
Let $G$ be a counterexample with fewest edges and fix $\kappa = \Delta(G) + 1$. For every proper subgraph $H$ of $G$, every component of $H$ is connected with maximum degree at most three and is not $3$-regular, thus every component of $H$ (hence $H$) has an acyclic edge coloring with at most $\kappa$ colors, which implies that $G$ is a $\kappa$-deletion-minimal graph. By \autoref{kappa=2} and \autoref{24edge}, the graph $G$ contains neither $1$-vertices nor $2$-vertices, thus $G$ is $3$-regular.
\end{proof}
\begin{theorem}[Hou \cite{2012arXiv1202.6129H}]\label{NMAD4}%
If $G$ is a graph with $\mad(G) < 4$, then $\chiup_{a}'(G) \leq \Delta(G) + 2$.
\end{theorem}
\begin{proof}%
Let $G$ be a counterexample with fewest edges and fix $\kappa = \Delta(G) + 2$. Since the hypothesis is deletion-closed, it follows that $G$ is a $\kappa$-deletion-minimal graph. By \autoref{kappa=2}, the graph $G$ is $2$-connected and $\delta(G) \geq 2$.
Since $\mad(G) < 4$, it follows that
\begin{equation}%
\sum_{v \in V(G)} (\deg(v) - 4) < 0.
\end{equation}

Assign the initial charge of every vertex $v$ to be $\deg(v) - 4$. We design appropriate discharging rules and redistribute charges among the vertices, such that the final charge of every vertex is nonnegative, which derives a contradiction.

{\bf The Discharging Rules:}
\begin{enumerate}[label= (R\arabic*)]%
\item Every $2$-vertex receives $1$ from each $6^{+}$-neighbor.
\item Every special $3$-vertex receives $1/2$ from each $5^{+}$-neighbor.
\item Every  normal $3$-vertex receives $1/3$ from each $5^{+}$-neighbor.
\end{enumerate}

By \autoref{2++edge}, every $2$-vertex is adjacent to two $6^{+}$-vertices, and then the final charge is $2 - 4 + 2 \times 1 = 0$. By \autoref{Good-3-vertex}, every special $3$-vertex is adjacent to two $5^{+}$-vertices, and then the final charge is $3 - 4 + 2 \times 1/2 = 0$. Every normal $3$-vertex is adjacent to three $5^{+}$-vertices, and then the final charge is $3 - 4 + 3 \times 1/3 = 0$. If $v$ is a $4$-vertex, then its final charge is equal to its initial charge zero.

Let $v$ be a $5$-vertex. By \autoref{Good-3-vertex}, if $v$ is adjacent to a special $3$-vertex, then $v$ is adjacent to at least three $4^{+}$-vertices, and then its final charge is at least $5 - 4 - 2 \times 1/2 = 0$. By \autoref{N_3_N}, if $v$ is not adjacent to any special $3$-vertex, then its final charge is at least $5 - 4 - 3 \times 1/3 = 0$.

Let $v$ be a $6^{+}$-vertex. If $v$ is adjacent to at least four $4^{+}$-vertices, then its final charge is at least $\deg(v) - 4 -(\deg(v) - 4) = 0$. So we may assume that $v$ is adjacent to at most three $4^{+}$-vertices. By \autoref{2+edge}, if $v$ is adjacent to some $2$-vertices and exactly three $4^{+}$-vertices, then its final charge is at least $\deg(v) - 4 - (\deg(v) - 5) \times 1 - 2 \times 1/2 = 0$. So we may assume that $v$ is not adjacent to any $2$-vertex. If $v$ is adjacent to some special $3$-vertices, then $v$ is adjacent to at least three $4^{+}$-vertices, and then its final charge is at least $\deg(v) - 4 - (\deg(v) - 3) \times 1/2 = (\deg(v) - 5)/2 > 0$. If all the $3^{-}$-vertices in $N_{G}(v)$ are normal $3$-vertices, then the final charge of $v$ is at least $\deg(v) - 4 - \deg(v) \times 1/3 = 2\deg(v)/3 - 4 \geq 0$.
\end{proof}

Two cycles are {\em adjacent} if they share a common edge, and are {\em intersecting} if they share a common vertex. Let $G$ be a plane graph, two faces are {\em adjacent} if they share a common edge, and are {\em intersecting} if they share a common vertex.

\begin{theorem}\label{NoAdjacent}%
Let $G$ be a planar graph without triangles adjacent to cycles of length $3$ and $4$. If every $5$-cycle has at most three edges contained in triangles, then $G$ admits an acyclic edge coloring with $\Delta(G) + 2$ colors.
\end{theorem}
\begin{proof}%
Let $G$ be a counterexample with fewest edges that has been embedded in the plane, and fix $\kappa = \Delta(G) + 2$. Since the hypothesis is deletion-closed, it follows that $G$ is a $\kappa$-deletion-minimal graph. By \autoref{kappa=2}, the graph $G$ is $2$-connected and the boundary of every face is a cycle. By the hypothesis, every $3$-face is adjacent to $5^{+}$-faces.

From Euler's formula, we have the following equality:
\begin{equation}%
\sum_{v \in V(G)} (\deg(v) - 4) + \sum_{f \in F(G)} (\deg(f) - 4) = - 8
\end{equation}

Assign the initial charge of every vertex $v$ to be $\deg(v) - 4$ and the initial charge of every face $f$ to be $\deg(f) - 4$. We design appropriate discharging rules and redistribute charges among vertices and faces, such that the final charge of every vertex and every face is nonnegative, which derives a contradiction.

{\bf The Discharging Rules:}
\begin{enumerate}[label= (R\arabic*)]%
\item Every $2$-vertex receives $1$ from each $6^{+}$-neighbor.
\item Every special $3$-vertex receives $1/2$ from each $5^{+}$-neighbor.
\item Every  normal $3$-vertex receives $1/3$ from each $5^{+}$-neighbor.
\item If $f$ is a $3$-face, then $f$ receives $1/3$ from adjacent $5^{+}$-faces passing through each of its incident edges.
\end{enumerate}

By \autoref{2++edge}, every $2$-vertex is adjacent to two $6^{+}$-vertices, and then the final charge is $2 - 4 + 2 \times 1 = 0$. By \autoref{Good-3-vertex}, every special $3$-vertex is adjacent to two $5^{+}$-vertices, and then the final charge is $3 - 4 + 2 \times 1/2 = 0$. Every normal $3$-vertex is adjacent to three $5^{+}$-vertices, and then the final charge is $3 - 4 + 3 \times 1/3 = 0$. If $v$ is a $4$-vertex, then its final charge is equal to its initial charge zero.

Let $v$ be a $5$-vertex. By \autoref{Good-3-vertex}, if $v$ is adjacent to a special $3$-vertex, then $v$ is adjacent to at least three $4^{+}$-vertices, and then its final charge is at least $5 - 4 - 2 \times 1/2 = 0$. If $v$ is not adjacent to any special $3$-vertex, then its final charge is at least $5 - 4 - 3 \times 1/3 = 0$ by \autoref{N_3_N}.

Let $v$ be a $6^{+}$-vertex. If $v$ is adjacent to at least four $4^{+}$-vertices, then its final charge is at least $\deg(v) - 4 -(\deg(v) - 4) = 0$. So we may assume that $v$ is adjacent to at most three $4^{+}$-vertices. By \autoref{2+edge}, if $v$ is adjacent to some $2$-vertices and exactly three $4^{+}$-vertices, then its final charge is at least $\deg(v) - 4 - (\deg(v) - 5) \times 1 - 2 \times 1/2 = 0$. So we may assume that $v$ is not adjacent to any $2$-vertex. If $v$ is adjacent to some special $3$-vertices, then $v$ is adjacent to at least three $4^{+}$-vertices, and then its final charge is at least $\deg(v) - 4 - (\deg(v) - 3) \times 1/2 = (\deg(v) - 5)/2 > 0$. If all the $3^{-}$-vertices in $N_{G}(v)$ are normal $3$-vertices, then the final charge of $v$ is at least $\deg(v) - 4 - \deg(v) \times 1/3 = 2\deg(v)/3 - 4 \geq 0$.

If $f$ is a $4$-face, then its final charge is zero. If $f$ is a $5$-face, then it is adjacent to at most three $3$-faces, thus its final charge is at least $5 - 4 - 3 \times 1/3 = 0$. If $f$ is a $6^{+}$-face, then its final charge is at least
\[
\deg(f) - 4 - \deg(f) \times \frac{1}{3} = \frac{2}{3} \deg(f) - 4 \geq 0.
\]

Therefore, the final charge of every vertex and every face is nonnegative, and then the sum of the final charges is nonnegative, which derives a contradiction.
\end{proof}

As immediate consequences of this theorem, we have the following corollaries.
\begin{corollary}%
Every planar graph $G$ without triangles adjacent to cycles of length from $3$ to $5$ admits an acyclic edge coloring with $\Delta(G) + 2$ colors.
\end{corollary}

\begin{corollary}[Wan and Xu \cite{Wan}]%
If $G$ is a planar graph without $i$-cycle adjacent to $j$-cycle for $i, j \in\{3, 4, 5\}$, then $\chiup_{a}'(G) \leq \Delta(G) + 2$.
\end{corollary}

\begin{corollary}[Hou \etal \cite{MR2891643}]%
If $G$ is a planar graph without $4$- and $5$-cycles, then $G$ admits an acyclic edge coloring with $\Delta(G) + 2$ colors.
\end{corollary}

\begin{corollary}[Hou \etal \cite{MR2891643}]%
If $G$ is a planar graph without $4$- and $6$-cycles, then $G$ admits an acyclic edge coloring with $\Delta(G) + 2$ colors.
\end{corollary}

\begin{corollary}[Fiedorowicz \cite{MR2915388}]%
If $G$ is a plane graph such that every vertex is contained in at most one $4^{-}$-face, then $G$ admits an acyclic edge coloring with $\Delta(G) + 2$ colors.
\end{corollary}

After this paper was submitted, Wang \etal published the following result which strengthens \autoref{NoAdjacent} a lot. Comparing their result with ours, they ignore the restriction on the $5$-cycles.

\begin{theorem}[\cite{MR3101746}]
If $G$ is a planar graph without a triangle adjacent to a $4$-cycle, then $\chiup_{a}'(G) \leq \Delta(G) + 2$.
\end{theorem}

We conclude this section by presenting a new result on the plane graph without intersecting triangles, which improves the result of Sheng-Wang \cite{MR2832148}. 

\begin{theorem}\label{NoIntersect}
If $G$ is a plane graph without intersecting triangles, then $\chiup_{a}'(G) \leq \Delta(G) + 3$.
\end{theorem}
Before proving \autoref{NoIntersect}, we require two further lemmas. By \autoref{Non4Regular} and the main result (every $4$-regular graph admits an acyclic edge coloring with six colors) in \cite{WangSW2012+}, the following lemma follows.

\begin{lemma}%
If $G$ is a $\kappa$-deletion-minimal graph with $\kappa \geq \Delta(G) + 2$, then $\Delta(G) \geq 5$.
\end{lemma}

A $(d_{1}, d_{2}, d_{3})$-cycle is a triangle with vertices of degree $d_{1}, d_{2}$ and $d_{3}$ respectively. The following lemma is not clearly stated, but it is contained in the proof of \cite{MR3044159}.
\begin{lemma}[Wang \etal \cite{MR3044159}]\label{NO444}%
Let $G$ be a $\kappa$-deletion-minimal graph. If $\kappa \geq \Delta(G) + 2$ and $\Delta(G) \geq 5$, then $G$ contains no $(4, 4, 4)$-cycles.
\end{lemma}

Now, we are ready to prove \autoref{NoIntersect}.
\begin{proof}[Proof of \autoref{NoIntersect}]%
Let $G$ be a counterexample with fewest edges and fix $\kappa = \Delta(G) + 3$. Since the hypothesis is deletion-closed, it follows that $G$ is a $\kappa$-deletion-minimal graph. By \autoref{kappa=2}, the graph $G$ is $2$-connected and the boundary of every face is a cycle.

From Euler's formula, we have the following equality:
\begin{equation}%
\sum_{v \in V(G)} (\deg(v) - 4) + \sum_{f \in F(G)} (\deg(f) - 4) = - 8
\end{equation}

Assign the initial charge of every vertex $v$ to be $\deg(v) - 4$ and the initial charge of every face $f$ to be $\deg(f) - 4$. We design appropriate discharging rules and redistribute charges among vertices and faces, such that the final charge of every vertex and every face is nonnegative, which derives a contradiction.

{\bf The Discharging Rules:}
\begin{enumerate}[label= (R\arabic*)]%
\item Every $2$-vertex receives $1$ from each $7^{+}$-neighbor.
\item Every special $3$-vertex receives $1/2$ from each $6^{+}$-neighbor.
\item Every  normal $3$-vertex receives $1/3$ from each $6^{+}$-neighbor.
\item Let $f$ be a $3$-face in $G$. If $f$ is incident with exactly one $5^{+}$-vertex, then $f$ receives $1$ from this $5^{+}$-vertex; if $f$ is incident with at least two $5^{+}$-vertices, then $f$ receives $1/2$ from each incident $5^{+}$-vertex.
\end{enumerate}

By \autoref{2++edge}, every $2$-vertex is adjacent to two $7^{+}$-vertices, and then the final charge is $2 - 4 + 2 \times 1 = 0$. By \autoref{Good-3-vertex}, every special $3$-vertex is adjacent to two $6^{+}$-vertices, and then the final charge is $3 - 4 + 2 \times 1/2 = 0$. Every normal $3$-vertex is adjacent to three $6^{+}$-vertices, and then the final charge is $3 - 4 + 3 \times 1/3 = 0$. If $v$ is a $4$-vertex, then its final charge is equal to its initial charge zero.

If $v$ is a $5$-vertex, then $v$ sends at most $1$ to incident $3$-face, and thus the final charge of $v$ is at least $5 - 4 - 1 = 0$.

Let $v$ be a $6$-vertex. By \autoref{2++edge}, the vertex $v$ is not adjacent to any $2$-vertex. Firstly, assume that $v$ is adjacent to a special $3$-vertex. By \autoref{Good-3-vertex}, the vertex $v$ is adjacent to at most two $3$-vertices, and then the final charge is at least $6 - 4 - 2 \times 1/2 -1 = 0$. Secondly, assume that all the $3$-neighbors of $v$ are normal. If $v$ sends $1$ to incident $3$-face, then this $3$-face contains exactly one $5^{+}$-vertex and it must be a $(4, 4, 6)$-face by \autoref{2++edge} and \autoref{Good-3-vertex}. By \autoref{NO44t}, if $v$ is incident with a $(4, 4, 6)$-face, then $v$ is adjacent to at most three $3$-vertices, and then the final charge is at least $6 - 4 - 3 \times 1/3 -1 = 0$. So we may assume that $v$ does not send $1$ to incident $3$-face. By \autoref{N_3_N}, the vertex $v$ is adjacent to at most four $3$-vertices, thus the final charge is at least $6 - 4 - 4 \times 1/3 -1/2 = 1/6$.

Let $v$ be a $7^{+}$-vertex. If $v$ is adjacent to at least five $4^{+}$-vertices, then its final charge is at least $\deg(v) - 4 - (\deg(v) - 5) - 1 = 0$. So we may assume that $v$ is adjacent to at most four $4^{+}$-vertices. By \autoref{2+edge}, if $v$ is adjacent to some $2$-vertices and exactly four $4^{+}$-vertices, then $v$ is adjacent to at most $\deg(v) - 6$ vertices of degree two, thus its final charge is at least $\deg(v) - 4 - (\deg(v) - 6) \times 1 - 2 \times 1/2 - 1 = 0$. So we may assume that $v$ is not adjacent to any $2$-vertex. By \autoref{Good-3-vertex}~\ref{3d} and \ref{3f}, if $v$ is adjacent to some special $3$-vertices, then $v$ is adjacent to at least four $4^{+}$-vertices, and then its final charge is at least $\deg(v) - 4 - (\deg(v) - 4) \times 1/2 - 1 = (\deg(v) - 6)/2 > 0$. So we assume that all the $3^{-}$-vertices in $N_{G}(v)$ are normal $3$-vertices. By \autoref{3+vertex}, if $v$ is incident with a $3$-face, then it is adjacent to at least one $4^{+}$-vertex, and then the final charge of $v$ is at least $\deg(v) - 4 - (\deg(v) - 1) \times 1/3 - 1 = (2\deg(v)-14)/3 \geq 0$. If $v$ is not incident with any $3$-face, then the final charge is at least $\deg(v) - 4 - \deg(v) \times 1/3 = (2\deg(v) - 12)/3 > 0$.

If $f$ is a $4^{+}$-face, then its final charge is $\deg(f) - 4 \geq 0$. By \autoref{2++edge}, \ref{3+vertex} and \autoref{NO444}, every $3$-face is incident with at least one $5^{+}$-vertex, then its final charge is nonnegative by (R4).

Therefore, the final charge of every vertex and every face is nonnegative, and then the sum of the final charges is nonnegative, which derives a contradiction.
\end{proof}
\section{Concluding remarks}\label{ConcludingR}
Note that \autoref{2+edge} does not provide any local structure on the $\kappa$-deletion-minimal graph $G$ when $\kappa = \Delta(G)$. Here, we consider one extremal case and prove the following result, which generalizes some results in \cite{Basavaraju2012, MR2921598}. The method is inspired by that used in \cite{Basavaraju2012}.
\begin{theorem}%
If $G$ is a graph with $\Delta(G) \geq 3$ and all the $3^{+}$-vertices are independent, then $\chiup_{a}'(G) = \Delta(G)$.
\end{theorem}
\begin{proof}%
Let $G$ be a counterexample with fewest edges and fix $\kappa = \Delta(G)$. Every proper subgraph with maximum degree at most two admits an acyclic edge coloring with at most three colors. Every proper subgraph with maximum degree at least three admits an acyclic edge coloring with $\kappa$ colors due to the minimality of $G$. Hence, the graph $G$ is a $\kappa$-deletion-minimal graph. By \autoref{kappa=2}, the graph $G$ is $2$-connected and the minimum degree is at least two. By \autoref{DegreeSum}, every $3^{+}$-vertex is a vertex with maximum degree.

Suppose that there exists an edge $xy$ with $\deg(x) = \deg(y) = 2$. By the $2$-connectivity of $G$, the edge $xy$ is not connected in any triangle, thus $G/xy$ is a simple graph and all the $3^{+}$-vertices are also independent, so $G/xy$ admits an acyclic edge coloring with $\kappa$ colors. It is easy to see that this edge coloring can be extended to an acyclic edge coloring of $G$ with $\kappa$ colors, a contradiction. Now we have shown that the graph $G$ is bipartite. Let $G$ be the bipartite graph with bipartition $X$ and $Y$, where $X$ is the collection of $\kappa$-vertices and $Y$ is the collection of $2$-vertices.
\begin{claim}\label{No_Common}%
For any acyclic edge coloring $\sigma$ of $G - xy$, we have $\mathcal{U}_{\sigma}(x) \cap \mathcal{U}_{\sigma}(y) = \emptyset$.
\end{claim}
\begin{proof}
Let $xy$ be an edge of $G$ with $x \in X$ and $y \in Y$. Let $N_{G}(x) = \{y, v_{1}, \dots, v_{\kappa -1}\}$ and $\sigma(xv_{i}) = i$ for $1 \leq i \leq \kappa -1$. By contradiction, assume that $\sigma(xv_{1}) = \sigma(yw) = 1$. The only candidate color $\kappa$ for $xy$ is not valid, thus there exists a $(1, \kappa, x, y)$-critical path with respect to $\sigma$. If there exists a vertex $v_{i}$ with $i \geq 2$ such that $\kappa \in \mathcal{U}_{\sigma}(v_{i})$, then exchanging the colors on $xv_{1}$ and $xv_{i}$ results in a new acyclic edge coloring $\sigma_{1}$ of $G - xy$, but now $\kappa$ is valid for $xy$ with respect to $\sigma_{1}$, a contradiction. So we have $\kappa \notin \mathcal{U}_{\sigma}(v_{i})$ for $2 \leq i \leq \kappa - 1$. Reassigning $\kappa$ to $xv_{2}$ results in another acyclic edge coloring $\sigma_{2}$ of $G - xy$. But now the color $2$ is valid for $xy$ with respect to $\sigma_{2}$.
\end{proof}

Let $x_{0}y_{0}$ be an edge of $G$ with $x_{0} \in X$ and $y_{0} \in Y$. The graph $G - x_{0}y_{0}$ admits an acyclic edge coloring $\phi$ with $\kappa$ colors. By \autoref{No_Common}, we have $\mathcal{U}(x_{0}) \cap \mathcal{U}(y_{0}) = \emptyset$. We may assume that $\mathcal{U}(x_{0}) = \{1, 2, \dots, \kappa - 1\}$ and $\mathcal{U}(y_{0}) = \{\kappa\}$. Recall that $\mathcal{U}(x) = \{1, 2, \dots, \kappa\}$ for every vertex $x$ in $X \setminus \{x_{0}\}$. Let $y_{0}x_{1}y_{1} \dots$ be the maximal $(\kappa, 2)$-path with respect to $\phi$. This path ends with an edge $x_{s}y_{s}$ which is colored with $2$, since the color $2$ appears at every vertex in $X$. Reassigning $\kappa$ to $x_{i+1}y_{i+1}$ and $2$ to $x_{i+1}y_{i}$ for $0 \leq i \leq s-1$, we obtain an acyclic edge coloring $\psi$ of $G - x_{0}y_{0}$. But $\mathcal{U}_{\psi}(x_{0}) = \{1, 2, 3, \dots, \kappa - 1\}$ and $\mathcal{U}_{\psi}(x_{0}) \cap \mathcal{U}_{\psi}(y_{0}) = \{2\}$, which contradicts \autoref{No_Common}.
\end{proof}

The concept of $\kappa$-deletion-minimal graph is defined by taking subgraphs. Analogously, we can define another type of minimal graphs by taking minors. A graph $G$ with maximum degree at most $\kappa$ is {\em $\kappa$-minimal} if $\chiup_{a}'(G) > \kappa$ and $\chiup_{a}'(H) \leq \kappa$ for every proper minor $H$ with $\Delta(H) \leq \Delta(G)$. Obviously, every proper subgraph of a $\kappa$-minimal graph admits an acyclic edge coloring with at most $\kappa$ colors, and then every $\kappa$-minimal graph is also a $\kappa$-deletion-minimal graph and all the properties of $\kappa$-deletion-minimal graphs are also true for $\kappa$-minimal graphs. Let $G/e$ denote the graph obtained by contracting the edge $e$ in $G$.
\begin{lemma}%
Let $G$ be a $\kappa$-minimal graph with $\kappa \geq \Delta(G) + 1$. If $v_{0}$ is a $2$-vertex of $G$, then $v_{0}$ is contained in a triangle.
\end{lemma}
\begin{proof}%
Let $N_{G}(v_{0}) = \{v, w\}$ and $e = vv_{0}$. By contradiction, suppose that $v$ and $w$ are nonadjacent. The graph $G/e$ is a simple graph with $\Delta(G/e) \leq \Delta(G)$, thus it admits an acyclic edge coloring $\phi$. We can extend $\phi$ by assigning a color in $C(v)$ to $vv_{0}$, and obtain an acyclic edge coloring of $G$, a contradiction. 
\end{proof}

\begin{lemma}
Let $G$ be a $\kappa$-minimal graph with $\kappa \geq \Delta(G) + 2$. If $v$ is a $3$-vertex in $G$, then every neighbor of $v$ is a $(\kappa-\Delta(G) + 3)^{+}$-vertex.
\end{lemma}
\begin{proof}%
By \autoref{3+vertex}, every neighbor of $v$ is a $(\kappa-\Delta(G) + 2)^{+}$-vertex. Suppose that $v$ is adjacent to a $(\kappa-\Delta(G) + 2)$-vertex $w$. By \autoref{Good-3-vertex}~\ref{3c}, the edge $wv$ is not contained in any triangle of $G$, thus the graph $G/wv$ is a simple graph. Note that the new vertex in graph $G/wv$ has degree $\kappa-\Delta(G) + 3$, thus according to \autoref{Good-3-vertex}~\ref{3b}, the graph $G/wv$ is a simple graph with maximum degree $\Delta(G)$. By the minimality of $G$, the simple graph $G/wv$ admits an acyclic edge coloring with at most $\Delta(G) + 2$ colors, but this edge coloring can be easily extended to an acyclic edge coloring of $G$ with at most $\Delta(G) + 2$ colors, a contradiction.
\end{proof}

\vskip 3mm \vspace{0.3cm} \noindent{\bf Acknowledgments.} The authors would like to thank the anonymous reviewers for their valuable comments and assistance on earlier drafts. The first author was supported by NSFC (11101125).

\end{document}